\documentclass[journal]{IEEEtran}

\usepackage{cite}
\usepackage[pdftex]{graphicx}
\usepackage{epstopdf}
\usepackage{amsmath}
\usepackage{amsfonts}
\usepackage{amsthm}
\theoremstyle{plain}
\newtheorem{thm}{Theorem}
\newtheorem{lem}{Lemma}
\newtheorem{prop}{Proposition}

\newtheorem{remark}{Remark}

\newtheorem{cla}{Claim}
\usepackage{algorithm}
\usepackage{algorithmic}
\usepackage{multirow} 
\usepackage{booktabs}

\usepackage{array}
\usepackage{graphicx} 
\ifCLASSOPTIONcompsoc
\usepackage[caption=false,font=normalsize,labelfont=sf,textfont=sf]{subfig}
\else
\usepackage[caption=false,font=footnotesize]{subfig}
\fi
\usepackage{color}
\usepackage{stfloats}
\usepackage{bm}
\usepackage{amsthm,amsmath,amssymb}
\usepackage{mathrsfs}
\allowdisplaybreaks[4]
\renewcommand{\arraystretch}{1.3}

\usepackage{nomencl}
\usepackage{ifthen}
\renewcommand{\nomgroup}[1]{%
\ifthenelse{\equal{#1}{A}}{\item[\emph{\textbf{Abbreviations}}]}{%
\ifthenelse{\equal{#1}{B}}{\item[\emph{\textbf{Sets and Indices}}]}{%
\ifthenelse{\equal{#1}{C}}{\item[\emph{\textbf{Parameters}}]}{%
\ifthenelse{\equal{#1}{D}}{\item[\emph{\textbf{Decision Variables}}]}{%
\ifthenelse{\equal{#1}{E}}{\item[\emph{\textbf{XX}}]}
				}
			}
		}
	}
}

\makenomenclature
\usepackage{verbatim}
\usepackage{makecell}
\usepackage{threeparttable}
\usepackage[figuresright]{rotating}

\begin{document}
\title{Stochastic-Robust Planning of \\Networked Hydrogen-Electrical Microgrids:\\ A Study on Induced Refueling Demand}

\author{Xunhang~Sun,~\IEEEmembership{Graduate~Student~Member,~IEEE},
        Xiaoyu~Cao,~\IEEEmembership{Member,~IEEE},
        Bo~Zeng,~\IEEEmembership{Member,~IEEE},
        Qiaozhu~Zhai,~\IEEEmembership{Member,~IEEE},
        Tamer~Ba\c{s}ar,~\IEEEmembership{Life~Fellow,~IEEE},
        and~Xiaohong~Guan,~\IEEEmembership{Life~Fellow,~IEEE}
\thanks{This work was partially supported by the National Key Research and Development Program of China under Grant 2022YFA1004600, and by the National Natural Science Foundation of China under Grant 62373294, Grant 62192752, and Grant 62103321. \emph{(Corresponding author: Xiaoyu Cao.)}}%
\thanks{Xunhang Sun is with the School of Automation Science and Engineering and the Ministry of Education Key Laboratory for Intelligent Networks and Network Security, Xi’an Jiaotong University, Xi’an 710049, Shaanxi, China, and also with the Coordinated Science Laboratory, University of Illinois Urbana-Champaign, Urbana, IL 61801 USA (e-mail: xhsun@sei.xjtu.edu.cn).}%
\thanks{Xiaoyu Cao and Qiaozhu Zhai are with the School of Automation Science and Engineering and the Ministry of Education Key Laboratory for Intelligent Networks and Network Security, Xi’an Jiaotong University, Xi’an 710049, Shaanxi, China, and also with the Smart Integrated Energy Department, Sichuan Digital Economy Industry Development Research Institute, Chengdu 610037, Sichuan, China (e-mail: cxykeven2019@xjtu.edu.cn; qzzhai@sei.xjtu.edu.cn).}%
\thanks{Bo Zeng is with the Department of Industrial Engineering and the Department of Electrical and Computer Engineering, University of Pittsburgh, Pittsburgh, PA 15106 USA (e-mail: bzeng@pitt.edu).}%
\thanks{Tamer Ba\c{s}ar is with the Coordinated Science Laboratory, University of Illinois Urbana-Champaign, Urbana, IL 61801 USA (e-mail: basar1@illinois.edu).}%
\thanks{Xiaohong Guan is with the School of Automation Science and Engineering and the Ministry of Education Key Laboratory for Intelligent Networks and Network Security, Xi’an Jiaotong University, Xi’an 710049, Shaanxi, China, and also with the Center for Intelligent and Networked Systems, Department of Automation, Tsinghua University, Beijing 100084, China (e-mail: xhguan@xjtu.edu.cn).}%
}


\maketitle

\begin{abstract}
Hydrogen-electrical microgrids are increasingly assuming an important role on the pathway toward decarbonization of energy and transportation systems. This paper studies networked hydrogen-electrical microgrids planning (NHEMP), considering a critical but often-overlooked issue, i.e., the demand-inducing effect (DIE) associated with infrastructure development decisions. Specifically, higher refueling capacities will attract more refueling demand of hydrogen-powered vehicles (HVs). To capture such interactions between investment decisions and induced refueling demand, we introduce a decision-dependent uncertainty (DDU) set and build a trilevel stochastic-robust formulation. The upper-level determines optimal investment strategies for hydrogen-electrical microgrids, the lower-level optimizes the risk-aware operation schedules across a series of stochastic scenarios, and, for each scenario, the middle-level identifies the ``worst'' situation of refueling demand within an individual DDU set to ensure economic feasibility. Then, an adaptive and exact decomposition algorithm, based on Parametric Column-and-Constraint Generation (PC\&CG), is customized and developed to address the computational challenge and to quantitatively analyze the impact of DIE. Case studies on an IEEE exemplary system validate the effectiveness of the proposed NHEMP model and the PC\&CG algorithm. It is worth highlighting that DIE can make an important contribution to the economic benefits of NHEMP, yet its significance will gradually decrease when the main bottleneck transits to other system restrictions.
\end{abstract}


\IEEEpeerreviewmaketitle

\printnomenclature

\nomenclature[A]{HV}{Hydrogen-powered vehicle}%
\nomenclature[A]{NHEMP}{Networked hydrogen-electrical microgrids planning}%
\nomenclature[A]{DDU/DIU}{Decision-dependent/-independent uncertainty}%
\nomenclature[A]{PC\&CG}{Parametric column-and-constraint generation}%
\nomenclature[A]{BC\&CG}{Benders column-and-constraint generation}%
\nomenclature[A]{C\&CG}{Column-and-constraint generation}%
\nomenclature[A]{HD}{Hydrogen dispenser}%
\nomenclature[A]{DIE}{Demand-inducing effect}%
\nomenclature[A]{RES}{Renewable energy source}%
\nomenclature[A]{ESS}{Electrical storage system}%
\nomenclature[A]{HSS}{Hydrogen storage system}%
\nomenclature[A]{CapEx}{Capital expenditure}%
\nomenclature[A]{OPEX}{Operational expense}%
\nomenclature[A]{ASE}{Annualized system expense}%
\nomenclature[A]{PDN}{Power distribution network}%
\nomenclature[A]{WT}{Wind turbine}%
\nomenclature[A]{PV}{Photovoltaic}%
\nomenclature[A]{ELZ}{Electrolyzer}%
\nomenclature[A]{HT}{Hydrogen tank}%
\nomenclature[A]{BB}{Battery bank}%
\nomenclature[A]{LP}{Linear program}%
\nomenclature[A]{MIP}{Mixed-integer program}%
\nomenclature[A]{BS}{Basic solution}%
\nomenclature[A]{BOS}{Basic optimal solution}%

\nomenclature[B]{$\Lambda$}{Candidate node set for hydrogen-electrical microgrids}%
\nomenclature[B]{$\mathcal{S}/s$}{Set/index of stochastic scenarios}%
\nomenclature[B]{$\mathcal{Z}/z$}{Set/index of refueling zones}%
\nomenclature[B]{$\mathcal{T}/t$}{Set/index of operating periods}
\nomenclature[B]{$(s,t)$}{Time index of operating period $t$ in scenario $s$}
\nomenclature[B]{$\mathcal{L}/(i,j)$}{Set/index of distribution lines}%
\nomenclature[B]{$\mathcal{N}/i,j,l$}{Set/index of nodes}%
\nomenclature[B]{$k$}{Index of candidate facilities}%
\nomenclature[B]{$\Omega_{\rm res}$}{Set of RES}%
\nomenclature[B]{$\Omega_{\rm hss}$}{Set of HSS}%
\nomenclature[B]{$\Omega_{\rm ess}$}{Set of ESS}%
\nomenclature[B]{$\mathcal{C}(i)$}{Set of children nodes of node $i$}%

\nomenclature[C]{$\Delta_t$}{Time interval of operating period}%
\nomenclature[C]{$\sigma$}{Scale factor from one typical day to one year}%
\nomenclature[C]{$c^{\rm i\&m}_{\rm hem}/c^{\rm i\&m}_{k}$}{Annualized summation of investment cost and maintenance cost of hydrogen-electrical microgrid/component $k$}%
\nomenclature[C]{$\iota_{\rm pl}$}{Penalty prices for unmet electric loads}%
\nomenclature[C]{$x_{k,i}^{\rm{max}}/x_{k,i}^{\rm{min}}$}{Max-/minimun number of component $k$ allowed to install at node $i$}%
\nomenclature[C]{$n_{i}^{\rm{max}}/n_{i}^{\rm{min}}$}{Max-/minimun number of HDs allowed to install at node $i$}%
\nomenclature[C]{$\overline{P}_{k}$}{Unit capacity of component $k$}%
\nomenclature[C]{$\overline{E}_{k}$}{Unit energy capacity of component $k$ in ESSs}%
\nomenclature[C]{$P_{k,{\rm ch}}^{\rm{max}}/P_{k,{\rm dis}}^{\rm{max}}$}{Maximum charge/discharge power of ESS $k$}%
\nomenclature[C]{$DoD_k$}{Maximum depth of discharge of ESS $k$}%
\nomenclature[C]{$\eta^{\rm ch}_k/\eta^{\rm dis}_k$}{Charge/discharge efficiency of ESS $k$}%
\nomenclature[C]{$\eta_{\rm elz}$}{Efficiency of electrolyzers}%
\nomenclature[C]{$LHV_{\rm H_2}$}{Low heating value of hydrogen}%
\nomenclature[C]{$\kappa_{k}$}{Degradation price of ESS $k$}%
\nomenclature[C]{$U_i^{\rm{max}}/U_i^{\rm{min}}$}{Voltage magnitude upper/lower bound at node $i$}%
\nomenclature[C]{$S_{\rm mv}^{\rm{max}}/S_{{\rm lv},i}^{\rm{max}}$}{Capacity limit of medium/low-voltage substation}%
\nomenclature[C]{$r_{ij}/x_{ij}$}{Resistance/reactance of line $(i,j)$}%
\nomenclature[C]{$PD_{i}^{s,t}$}{Active electric loads at node $i$ at $(s,t)$}%
\nomenclature[C]{$\delta_{k,i}^{s,t}$}{Capacity time-varing factor of RES $k$ restricted by natural resources at node $i$ at $(s,t)$}%
\nomenclature[C]{$\varphi_{k}^{\rm{max}}/\varphi_{k}^{\rm{min}}$}{Max-/minimum power factor angle of RES $k$}%
\nomenclature[C]{$\varphi_{{\rm pl},i}^{s,t}$}{Power factor angle of electric loads at node $i$ at $(s,t)$}%
\nomenclature[C]{$\varrho_{\rm imp}^{t}/\varrho_{\rm g}$}{Utility prices for power/hydrogen procurement}%
\nomenclature[C]{$\rho_{\rm e}^{t}/\rho_{\rm g}$}{Retail electricity/hydrogen prices}%
\nomenclature[C]{$SR$}{Service rate of HD}%
\nomenclature[C]{$\phi_{\rm ht}$}{Dissipation factor of hydrogen tanks}%
\nomenclature[C]{$\pi_{s}$}{Probability of scenario $s$}%
\nomenclature[C]{$\overline{G}_{{\rm pur},z}^{t}$}{Hydrogen purchase limitation of zone $z$ at $t$}%
\nomenclature[C]{$\xi_{z}^{s,t}$}{Refueling demand of HVs at zone $z$ at $(s,t)$}%

\nomenclature[D]{$x_{k,i}$}{Number of component $k$ at node $i$}%
\nomenclature[D]{$n_{i}$}{Number of HDs at node $i$}%
\nomenclature[D]{$u_{i}$}{Binary, 1 if a hydrogen-electrical microgrid is deployed at node $i$, and 0 otherwise}%
\nomenclature[D]{$p_{k,i}^{s,t}/q_{k,i}^{s,t}$}{Active/reactive power output of RES $k\in\Omega_{\rm res}$ at node $i$ at $(s,t)$}%
\nomenclature[D]{$pc_{k,i}^{s,t}/pd_{k,i}^{s,t}$}{Charging/discharging power of ESS $k$ at node $i$ at $(s,t)$}%
\nomenclature[D]{$p_{{\rm elz},i}^{s,t}/g_{{\rm elz},i}^{s,t}$}{Power input/hydrogen outflow of electrolyzers at node $i$ at $(s,t)$}%
\nomenclature[D]{$loh_{{\rm ht},i}^{s,t}$}{Level of hydrogen of hydrogen tanks at node $i$ at $(s,t)$}%
\nomenclature[D]{$soc_{k,i}^{s,t}$}{State-of-charge of ESS $k$ at node $i$ at $(s,t)$}%
\nomenclature[D]{$U_i^{s,t}$}{Magnitude of voltage at node $i$ at $(s,t)$}%
\nomenclature[D]{$fp_{ij}^{s,t}/fq_{ij}^{s,t}$}{Active/reactive power flow on line $(i,j)$ at $(s,t)$}%
\nomenclature[D]{$p_{\rm mv}^{s,t}/q_{\rm mv}^{s,t}$}{Active/reactive power via medium-voltage substation at $(s,t)$}%
\nomenclature[D]{$p_{{\rm lv},i}^{s,t}/q_{{\rm lv},i}^{s,t}$}{Active/reactive power via low-voltage substation at node $i$ at $(s,t)$}%
\nomenclature[D]{$pl_{i}^{s,t}/ls_{i}^{s,t}$}{Met/unmet electric loads at node $i$ at $(s,t)$}%
\nomenclature[D]{$gl_{i}^{s,t}/ul_{z}^{s,t}$}{Met/unmet HV refueling demand at node $i$/zone $z$ at $(s,t)$}%
\nomenclature[D]{$g_{{\rm pur},i}^{s,t}$}{Hydrogen purchased from market at node $i$ at $(s,t)$}%

\section{Introduction}
\subsection{Background and Literature Review}
\IEEEPARstart{O}{nly} by reaching global carbon neutrality in the middle of the 21st century can it be possible to control global warming within 1.5°C, thereby averting the extreme hazards caused by climate change \cite{ipcc1,wang2023accelerating,ZHANG2023113745}. To achieve carbon neutrality, it is imperative that the transportation sector, responsible for approximately 23\% of social carbon emissions, undergoes a profound transformation towards the low-carbon paradigm \cite{iea3}. Compared to carbon-intensive fossil-fuel vehicles, the hydrogen-powered vehicles (HVs), e.g., cars, vans, buses, and trucks that use hydrogen as fuel, are garnering increasing public attention, owing to their advantages of zero carbon emissions \cite{9755957,4168013}. With the integration of HVs into the transportation system, it has advanced the infrastructure for hydrogen production, storage, transportation, and refueling \cite{iea2-1}.  Also, some related regulations and standards have been developed \cite{iea2-1}.

However, the rapid proliferation of HVs raises enormous challenges to the energy sector for hydrogen supply. On the one hand, it is anticipated that there will be a substantial surge in hydrogen refueling demand. According to International Energy Agency, the global demand for hydrogen in road transport is projected to increase by 123 times in 2030 compared to 2022 \cite{iea2-1}. In light of such huge increase, the existing hydrogen supply capacity falls short. On the other hand, traditionally, hydrogen is mainly produced via centralized fossil-based ways, e.g., steam methane reforming and coal gasification, and then stored and delivered to end-users through transportation and distribution networks \cite{sinigaglia2017production,nikolaidis2017comparative,ABDALLA2018602}. Although this procedure is profitable, these hydrogen production measures could come with a poor environmental impact attribute to inherent carbon emissions \cite{mingolla2024effects}, which hinders progress towards achieving carbon neutrality. 

To support carbon-free hydrogen supply, it would be a viable approach to seek for dispersed generation of renewable energy sources (RESs) over the power distribution network (PDN), which is the critical infrastructure coupled with the urban transportation system. The interdependence facilities of RES generation, on-site hydrogen production and storage, as well as HVs' refueling can constitute the hydrogen-electrical microgrid for decarbonizing both the energy and transportation sectors \cite{hydro11}. In particular, hydrogen-electrical microgrids exploit the ``otherwise-curtailed'' renewable electricity generated from wind and solar energy installations to produce hydrogen to supply HVs by means of the power-to-hydrogen technology, i.e., water electrolysis \cite{WANG2024114779,shao2023risk}. It is noteworthy to mention that such refueling procedure for HVs is carbon-free. On the other hand, the RES outputs could be effectively harnessed to support the flexible operation of PDN \cite{chen2024risk,10354055}. Also, by incorporating the electrical and hydrogen storage systems, a clustering of HE microgrids could actively participate in demand response programs, thereby enhancing their economic benefits \cite{hydro2}.

Electrolytic production of hydrogen, however, suffers from high financial costs \cite{hydro7,pan2024feasibility,UECKERDT2024104}. As mentioned in Ref. \cite{nikolaidis2017comparative}, the hydrogen production cost of renewable electrolysis is 2.45--11.19 and 3.81--17.37 times higher than traditional steam methane reforming and coal gasification, respectively. Hence, a proper deployment of hydrogen-electrical microgrids is crucial for meeting increasing HV refueling demand and enhancing economic efficiency.

The primary challenge of hydrogen-electrical microgrids planning is the uncertain operational risks embedded in both the supply side (e.g., intermittent RES outputs) and the demand side (e.g., stochastic HV refueling demand and electric loads). How to make effective siting and sizing decisions for hydrogen-electrical microgrids within highly uncertain environments is an important line of current research. Based on the statistical characteristics of uncertain risks, Ref. \cite{hydro11} has proposed a two-stage stochastic programming method for hydrogen-electrical microgrids planning, using scenarios that capture the uncertainties. Further, Ref. \cite{9999551} has presented a multistage stochastic extension for systems' dynamic evolution, considering both the strategic and operating uncertainties under multiple timescales. Besides, Refs. \cite{8972566,wang2020robust} have exploited uncertainty sets to handle the uncertainties associated with hydrogen-electrical microgrids operation, and then proposed planning methods based on robust optimization. Moreover, Refs. \cite{yang2024distributionally,10538062} have introduced the distributionally robust optimization in the planning of hydrogen-electrical microgrids. Of note, in these existing works, the operational risks are generally modeled as \textit{decision-independent uncertainties (DIUs)}, or exogenous uncertainties, i.e., they are not affected by the investment decision choice \cite{hellemo2018decision}. More specifically, the DIU factors are often assumed to demonstrate a fixed (static) nature of randomness, e.g., taking values according to a pre-defined stochastic distribution (in stochastic programming) \cite{hydro11,9999551}, an uncertainty set (in robust optimization) \cite{8972566,wang2020robust}, or a set of distributions (in distributionally robust optimization) \cite{yang2024distributionally,10538062}.

Practically, it is often the case that more constructions and capacities attract higher traffic or demand, a phenomenon referred to as \textit{demand-inducing effect (DIE)} \cite{NOLAND20021,lee1999induced}. In fact, DIE is quite relevant in planning and building transportation infrastructure. Ref. \cite{MAHMUTOGULLARI2023173} has investigated the refueling station location problem, showing that without a DIE consideration, the planning scheme might result in a potential reduction of up to 17.07\% in the serviced traffic flows. Ref. \cite{li2017market} has mentioned that a 10\% increase in the number of public charging stations would increase the sales of electric vehicles by about 8\%. In hydrogen-electrical microgrids planning, it is also anticipated that the DIE, i.e., HV refueling demand that would be induced by hydrogen refueling capacities, would also be strong. Nevertheless, it is of note that the existing literature regarding the planning of hydrogen-electrical microgrds lacks attention on DIE, and there remains a dearth of a systematic study on induced refueling demand. 

The traditional DIU-based model cannot capture the system endogeneity (dynamics) associated with the DIE for HVs. Fortunately, the \textit{decision-dependent uncertainty (DDU)}, also known as endogenous uncertainty, i.e., the uncertain factor that is substantially affected by the siting and sizing decision choices \cite{zeng2022two,doi:10.1137/17M1110560,doi:10.1137/22M1502082}, fits better into the description of refueling demand in this situation. Note that the planning schemes of hydrogen-electrical microgrids without a DDU consideration are inadequate in handling the uneven demand distribution incurred by DIE, which would reduce the economic benefits. Therefore, in this study, we investigate the networked hydrogen-electrical microgrids planning (NHEMP) problem with the DDU modeling of HV refueling demand, to facilitate the low-carbon and economic development of the synergistic energy-transportation system.

It is worth highlighting that the inclusion of DDU revolutionarily changes the decision-making paradigm. In the traditional DIU-based optimization, the description of uncertainties is static. The decision-maker directly makes her optimal decision based on a fixed model of the DIU factors, e.g., a pre-defined probability distribution, an uncertainty set, or a set of distributions. While in the DDU setting, she has to further take into consideration the change of the uncertainties caused by her decisions, leading to \textit{mutual interactions between decisions and uncertain factors}. How to make sound decisions efficiently under such a complex DDU circumstance is another important concern addressed in our work.

\subsection{Contributions and Paper Structure}
This paper presents a stochastic-robust planning method for networked hydrogen-electrical microgrids, taking DIE into consideration. To ensure the economic feasibility of NHEMP with respect to the uncertain future HV refueling demand, we adopt a trilevel $\min-\max-\min$ formulation that inherently embodies the idea of robust optimization. The upper-level problem determines the siting and sizing strategies for hydrogen-electrical microgrids with minimum annual capital expenditures (CapEx's) and maintenance costs. The middle-level identifies the “worst” situation of refueling demand by adopting DDU sets. The DDU sets can capture the influence of investment decisions to refueling demand, which analytically reflects the DIE. Then, in the lower-level problem, the multi-period operational schedule of hydrogen-electrical microgrids associating with HV refueling actions is optimized to evaluate the profitability of the deployment scheme. Besides the DDU description, some other regular uncertain factors, i.e., the random variation of RES outputs and electric loads, are also considered in our NHEMP model by using a series of stochastic scenarios. In a summarized form, the decision framework is depicted in Fig. \ref{framework}. 

\begin{figure}[]
	\centering
	\includegraphics[scale=0.44]{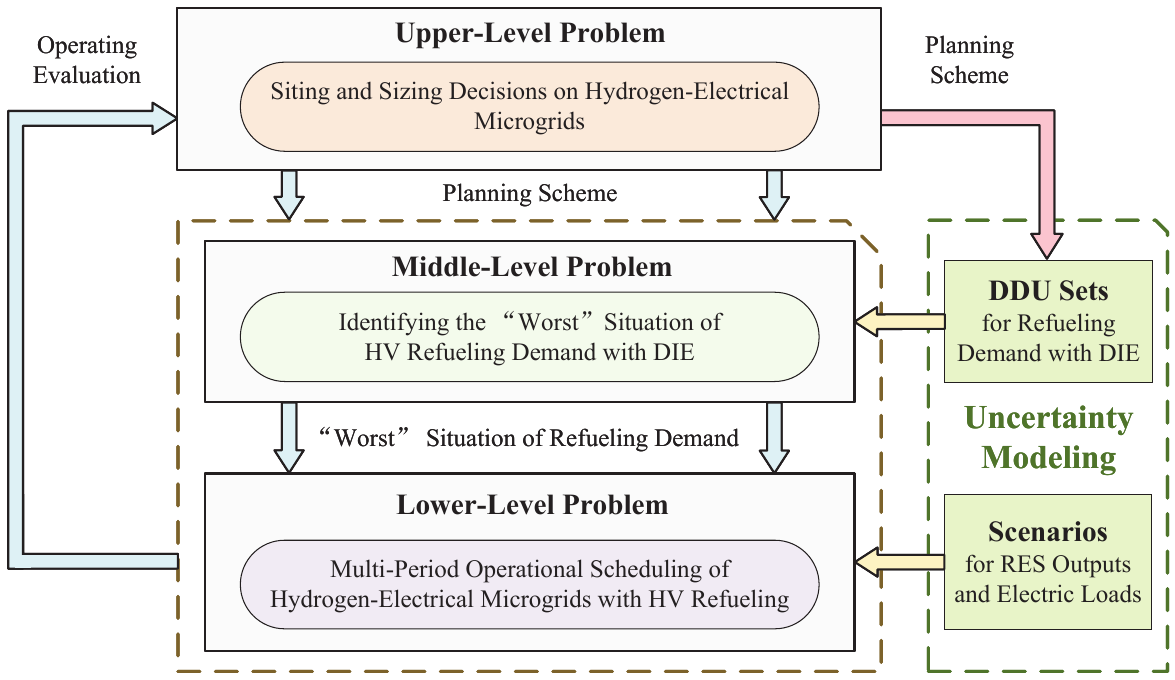}
	\caption{Framework of the trilevel stochastic-robust NHEMP problem.}
	\label{framework}
\end{figure}

The resulting NHEMP formulation has \textit{mutual interactions between investment strategies (upper-level decisions) and refueling demand (environment parameters in the middle-level)}. The popular Column-and-Constraint Generation (C\&CG) algorithm \cite{zeng2013solving} for classical DIU-based $\min-\max-\min$ problems cannot handle it because the DDU sets dynamically vary with the upper-level decisions. To address the computational intractability, we analyze the differences between the DDU and DIU sets, and derive a set of structural properties. Following them, an adaptive decomposition algorithm based on Parametric C\&CG (PC\&CG) is tailored and developed, which demonstrates a superior and exact solution capacity.

In comparison to existing literature, the contributions of this paper can be summarized as below:
\begin{enumerate}
	\item The often-overlooked DIE is taken into consideration in the context of hydrogen-electrical microgrids planning, and the corresponding refueling demand is captured by a DDU set. Note that DIE, if exists, cannot be ignored as it can generate a critical impact in system planning. Also, our modeling scheme provides a rather general tool to investigate DIE in other practical systems. 
	\item A trilevel stochastic-robust formulation is developed to ascertain the profitability of NHEMP, where the refueling demand is modeled by DDU sets, while the other DIU factors (i.e., intermittent RES outputs and invariant electric loads) are represented by a series of scenarios.
	\item To overcome the computational challenges associated with the stochastic-robust program with DDU, a PC\&CG-based adaptive decomposition algorithm is customized and implemented that guarantees to generate an exact solution.
	\item Numerical results demonstrate the value of DIE, the flexibility of DDU modeling, as well as the strong solution capacity of the customized PC\&CG-based adaptive decomposition algorithm.
\end{enumerate}

The remainder of this paper is organized as follows. Section \ref{formulation} proposes a trilevel stochastic-robust formulation of NHEMP considering DDU of refueling demand. Section \ref{stochastic} customizes a solution method based on PC\&CG. Then, we carry out case studies in Section \ref{case} to substantiate the effectiveness of our NHEMP model and customized PC\&CG algorithm. Some key findings are highlighted and discussed in Section \ref{discussion}. Finally, in Section \ref{conclusion}, conclusions are drawn, and several interesting directions are pointed out for future research.

\section{Problem Formulation}\label{formulation}
The goal of the NHEMP problem is to conduct optimal deployment of hydrogen-electrical microgrids (as shown in Fig. \ref{fig1}) within the complex operation environment. Particularly, the DIE of HVs are taken into consideration. Our NHEMP model features a trilevel structure with embedded DDUs, as presented in Fig. \ref{framework}.

\begin{figure}[]
	\centering
	\subfloat[Carbon-free town.]{
		\label{fig1-1}
		\includegraphics[width=0.48\textwidth]{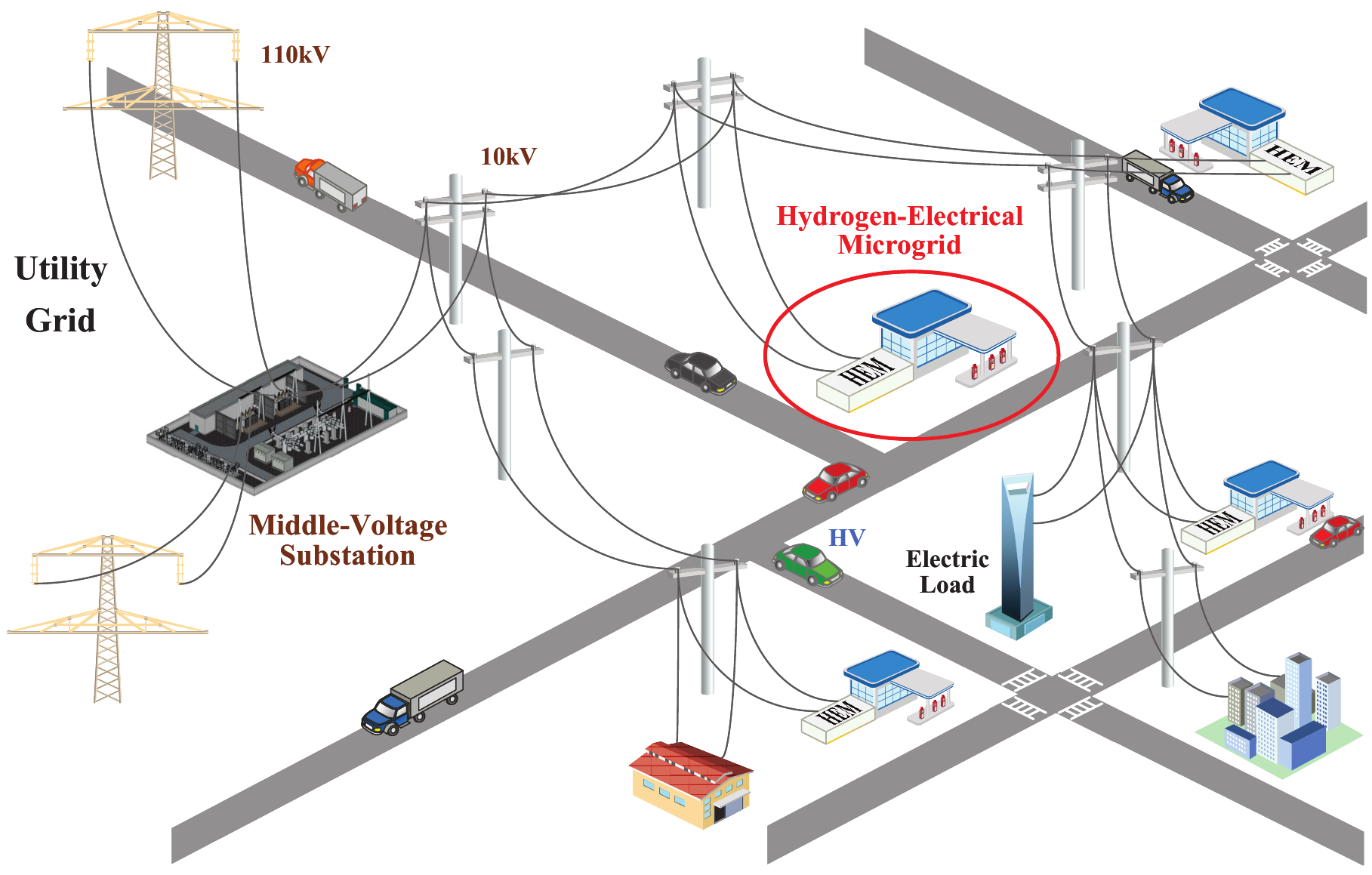}
	}\\
	\subfloat[Hydrogen-electrical microgrid.]{
		\label{fig1-2}
		\includegraphics[width=0.48\textwidth]{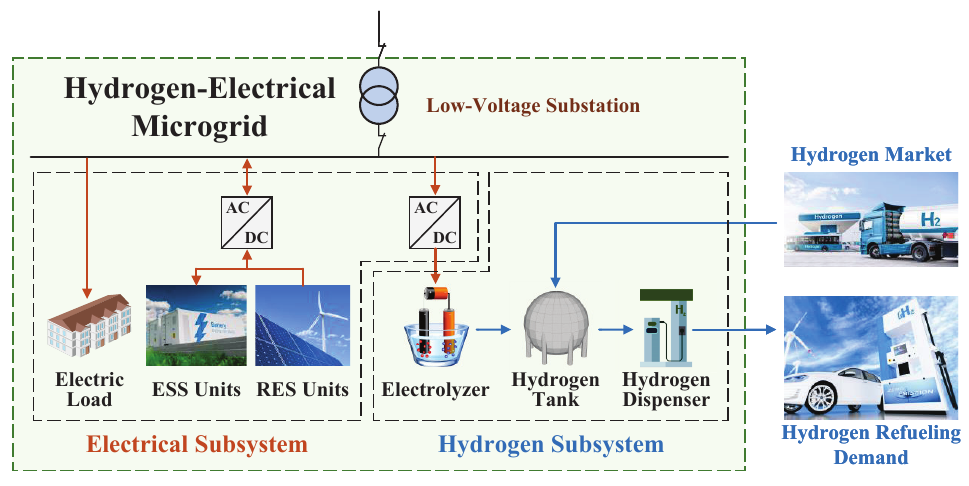}
	}
	\caption{Typical structure of carbon-free town with networked hydrogen-electrical microgrids.}
	\label{fig1}
\end{figure}

The PDN holds a radial topology, where $\mathcal{N}$ and $\mathcal{L}$ denote the sets of distribution nodes and lines, respectively. The medium-voltage substation, connected to the utility grid, is located at node 0. Due to geographical and architectural restrictions, only nodes in $\Lambda\subseteq\mathcal{N}^{+}=\{1,2,\cdots\}$ will be the candidates for possible deployment of hydrogen-electrical microgrids. We simplify the transportation system model. The town is partitioned into several mutually disjoint zones for local HV refueling. It is stipulated that HVs within each zone can only access one microgrid affiliated with the same zone for hydrogen refueling. Let $\Lambda_z$ denote the set of candidate nodes that can be deployed with hydrogen-electrical microgrids in zone $z\in\mathcal{Z}$.

\subsection{Trilevel Stochastic-Robust NHEMP Problem with DDU}
NHEMP considering DIE of HVs can be formulated as the following trilevel $\min-\max-\min$ problem with DDU in \eqref{Objective Function:1}--\eqref{branches}. Its target is to minimize the annualized system expenses (ASEs) under the ``worst-case'' refueling situation of HVs by making the optimal investment strategy and operating schedule for hydrogen-electrical microgrids, so as to verify the economic feasibility of the deployment scheme. 
\setlength{\arraycolsep}{-0em}
\begin{align}
	\nonumber\min~&\sum\limits_{i \in \Lambda } \left(c^{\rm i\& m}_{\rm hem} u_{i}+\sum\limits_{\substack{k \in \Omega_{\rm res}\cup\\ \Omega_{\rm ess}\cup\Omega_{\rm hss}}} c^{\rm i\& m}_{k} \overline{P}_{k}x_{k,i}+c^{\rm i\& m}_{\rm hd} n_{i}\right)\\
	&+\sum_{s\in\mathcal{S}}\pi_s\max_{\xi_{z}^{s,t}\in\Xi(n_{i})} \min~ Q_s(x_{k,i},n_{i},\xi_{z}^{s,t})\label{Objective Function:1}\\
	{\rm s.t.}~&\sum\limits_{i\in\Lambda_z }u_i=1,\quad \forall z\in\mathcal{Z} \label{cons2-1}\\
	\nonumber &x_{k,i}^{\rm{min}}u_i\leq x_{k,i}\leq x_{k,i}^{\rm{max}}u_i,\\
	&\quad\;\;\qquad\qquad\quad\quad \forall k\in\Omega_{\rm res}\cup\Omega_{\rm ess}\cup\Omega_{\rm hss},\forall i\in\Lambda \label{cons2-2}\\
	&n_{i}^{\rm{min}}u_i\leq n_{i}\leq n_{i}^{\rm{max}}u_i,\quad \forall i\in\Lambda\label{cons2-2-1}\\
	& \eqref{cons3-1}-\eqref{cons3-35},\quad \forall s\in\mathcal{S}\label{branches}
\end{align}

\subsubsection{Upper-Level Problem}
By choosing the optimal siting and sizing decisions for hydrogen-electrical microgrids, the objective of the upper-level problem is to minimize the annualized summation of CapEx's as well as the basic maintenance costs (Line 1 of \eqref{Objective Function:1}). To satisfy the hydrogen supply to the local area, and for the sake of simplicity, we assume that each zone is restricted to deploy only one hydrogen-electrical microgrid, as indicated by \eqref{cons2-1}. Besides, \eqref{cons2-2} and \eqref{cons2-2-1} show the capacity constraints for the hydrogen-electrical microgrids. Taking \eqref{cons2-2} for example, when $u_i=1$, $x_{k,i}\in[x_{k,i}^{\min},x_{k,i}^{\max}]$; otherwise, when $u_i=0$, $x_{k,i}=0$.

\subsubsection{Lower-Level Problem}
The lower-level problem contributes to obtaining a whole picture of the hydrogen-electrical microgrids' performance under varying operating conditions to assess the feasibility, reliablity, profitability, and flexiblity of the investment strategy. A set of \textit{stochastic scenarios} ($s\in\mathcal{S}$) \cite{stochastic1}, including random variation of RES outputs ($\delta_{k,i}^{s,t}$) and uncertain electric loads ($PD_i^{s,t}$), are introduced to describe the operational risks of hydrogen-electrical microgrids. The HVs' refueling demand ($\xi_z^{s,t}$) is excluded from the stochastic scenario, which will be illustrated in the middle-level problem.

Each scenario $s\in\mathcal{S}$ corresponds to a branch of the lower-level problem. The objective of branch $s$ is to minimize the operational expenses (OPEXs), i.e., power and hydrogen procurement costs, penalty costs associated with unmet electric demand, and degradation costs of ESSs, less the electrical and hydrogen revenues, as in \eqref{Objective Function:3}.
\setlength{\arraycolsep}{-0.3em}
\begin{eqnarray}
		&&\nonumber Q_s(x_{k,i},n_{i},\xi_{z}^{s,t})=\sigma \sum\limits_{t \in \mathcal{T}}\left(\varrho_{\rm imp}^{t} p_{\rm mv}^{s,t}+\varrho_{\rm g}\sum_{i\in\Lambda}g_{{\rm pur},i}^{s,t}\right)\Delta_t\\
		&&\nonumber\;\;+\sigma \sum\limits_{t \in \mathcal{T}}\Bigg[\iota_{\rm pl}\!\!\sum_{i \in \mathcal{N^+} }\!ls_i^{s,t}\!+\!\frac{1}{2}\!\sum_{i\in\Lambda}\!\sum_{k\in\Omega_{\rm ess}}\kappa_{k}\bigg(pc_{k,i}^{s,t}\!+\!pd_{k,i}^{s,t} \bigg) \Bigg] \Delta_t\\
		&&\;\;-\sigma\sum_{t\in\mathcal{T}}\Bigg(\rho^{t}_{\rm e}\sum_{i\in\mathcal{N}^+} pl_i^{s,t}+\rho_{\rm g}\sum_{i\in\Lambda}  gl_{i}^{s,t} \Bigg)\Delta_t\label{Objective Function:3}
\end{eqnarray}
The corresponding operational constraints are given below: 

\textbf{$\bullet$} \textit{Constraints for Electrical Subsystems:}
The varying outputs of RESs, e.g., photovoltaic (PV) arrays and wind turbines (WTs), are constrained by \eqref{cons3-1} and \eqref{cons3-2}. \eqref{cons3-3}--\eqref{cons3-9} are general operational constraints for electrical storage systems (ESSs), e.g., battery banks (BBs). The charging and discharging power are restricted by \eqref{cons3-3}. The state-of-charge is described by \eqref{cons3-7}-\eqref{cons3-9}. Note that \eqref{cons3-9} illustrates that the net charging capacities of ESSs should be zero after a daily charging-discharging cycle. \eqref{cons3-10} and \eqref{cons3-11} represent, respectively, the active and reactive power balance within each  microgrid. Moreover, \eqref{cons3-13} is the capacity constraint for low-voltage substations, which can be linearized by the technique in Ref. \cite{linear1}.
\setlength{\arraycolsep}{-0.3em}
\begin{eqnarray}
&&0\leq p_{k,i}^{s,t}\leq \delta_{k,i}^{s,t} \overline{P}_k x_{k,i},\quad \forall k \in \Omega_{\rm res},\forall i \in \Lambda, \forall t \in \mathcal{T}\label{cons3-1}\\
\nonumber&&p_{k,i}^{s,t}\tan \varphi_{k}^{\rm{min}} \leq q_{k,i}^{s,t}\leq p_{k,i}^{s,t} \tan \varphi_{k}^{\rm{max}} ,\\ &&\qquad\quad\qquad\qquad\qquad\qquad\qquad\quad\;\;\; \forall k \in \Omega_{\rm res},\forall i \in \Lambda,  \forall t \label{cons3-2}\\
&&0\leq pc_{k,i}^{s,t},pd_{k,i}^{s,t}\leq \overline{P}_k x_{k,i},\quad\forall k  \in \Omega_{\rm ess},\forall i, \forall t\label{cons3-3} \\
&&\nonumber soc_{k,i}^{s,t+1}=soc_{k,i}^{s,t}+\left(pc_{k,i}^{s,t} \eta^{\rm ch}_k-pd_{k,i}^{s,t}/\eta^{\rm dis}_k\right) \Delta_t,\\&&\qquad\quad\qquad\qquad\qquad\qquad\qquad\quad \forall k  \in \Omega_{\rm ess},\forall i \in \Lambda,  \forall t \label{cons3-7}\\
&&\nonumber \left(1-DoD_k\right)\overline{E}_k x_{k,i}\leq soc_{k,i}^{s,t}\leq \overline{E}_k x_{k,i},\\&&\quad\qquad\qquad\qquad\qquad\qquad\qquad\quad \forall k \in \Omega_{\rm ess},\forall i \in \Lambda,  \forall t \label{cons3-8}\\
&&\sum\limits_{t \in \mathcal{T}}\left(pc_{k,i}^{s,t} \eta_{k}^{\rm ch}-pd_{k,i}^{s,t}/\eta_{k}^{\rm dis}\right)=0,\;\forall k  \in \Omega_{\rm ess},\forall i \in \Lambda, \forall t \label{cons3-9}\\
\nonumber&&\sum\limits_{k \in\Omega_{\rm res}}p_{k,i}^{s,t}+\sum\limits_{k \in\Omega_{\rm ess}}\left(pd_{k,i}^{s,t}-pc_{k,i}^{s,t}\right)+p_{{\rm lv},i}^{s,t}\\
&&\qquad\qquad\qquad\qquad\quad\qquad=pl_i^{s,t}+p_{{\rm elz},i}^{s,t},\quad \forall i \in \Lambda,  \forall t \label{cons3-10}\\
&&\sum\limits_{k \in \Omega_{\rm res}}q_{k,i}^{s,t}+q_{{\rm lv},i}^{s,t}=pl_i^{s,t}\tan\varphi_{{\rm pl},i}^{s,t},\quad\forall i \in \Lambda, \forall t \label{cons3-11}\\
&& pl_i^{s,t}+ls_i^{s,t} = PD_i^{s,t},\quad \forall i \in \mathcal{N}^+, \forall t \label{cons3-12}\\
&&0 \leq pl_i^{s,t},ls_i^{s,t} \leq PD_i^{s,t},\quad \forall i \in \mathcal{N}^+, \forall t \label{cons3-12-1}\\
&&\left(p_{{\rm lv},i}^{s,t}\right)^2+\left(q_{{\rm lv},i}^{s,t}\right)^2\leq \left(S_{{\rm lv},i}^{\rm{max}}\right)^2,\quad \forall i \in \Lambda, \forall t \label{cons3-13}
\end{eqnarray}

\textbf{$\bullet$} \textit{Constraints for Hydrogen Subsystems:}
Electrolyzers (ELZs) and hydrogen tanks (HTs) constitute the hydrogen storage system (HSS). \eqref{cons3-14} below captures the electricity-hydrogen transition process of  ELZs, and \eqref{cons3-16} describes the limit on hydrogen production. $x_{{\rm elz},i}$ denotes the number of ELZs deployed at node $i$. \eqref{cons3-20} describes the massive balance of HTs, taking into account the dissipation factor $\phi_{ht}$. The hydrogen purchased from the local hydrogen market is constrained by \eqref{cons3-20-1}. The capacity limitation of HTs is given in \eqref{cons3-21}. $x_{{\rm ht},i}$ indicates the invested number of HTs. The met ($gl_i^{s,t}$) and unmet ($ul_z^{s,t}$) refueling demand of HVs are governed by \eqref{cons3-22} and \eqref{cons3-22-1}, and $gl_i^{s,t}$ is also constrained by \eqref{consHD} due to the service rate ($SR$) of hydrogen dispensers (HDs). 
\begin{eqnarray}
&&g_{{\rm elz},i}^{s,t}=\frac{\eta_{\rm elz} p_{{\rm elz},i}^{s,t}}{LHV_{\rm H_2}},\quad \forall i \in \Lambda, \forall t\label{cons3-14} \\
&&0\leq p_{{\rm elz},i}^{s,t}\leq \overline{P}_{\rm elz}x_{{\rm elz},i} ,\quad \forall i \in \Lambda,  \forall t \label{cons3-16}\\
\nonumber&&loh_{{\rm ht},i}^{s,t+1}=loh_{{\rm ht},i}^{s,t}\left(1-\phi_{\rm ht}\right)\\
&&\qquad\qquad\qquad +\left(g_{{\rm elz},i}^{s,t}+g_{{\rm pur},i}^{s,t}-gl_{i}^{s,t}\right)\Delta_t,\;\forall i \in \Lambda, \forall t \label{cons3-20}\\
&&0\leq g_{{\rm pur},i}^{s,t}\leq\overline{G}_{{\rm pur},z}^{t}u_i,\quad \forall i \in \Lambda_z,z\in\mathcal{Z},  \forall t\label{cons3-20-1}\\
&&0\leq loh_{{\rm ht},i}^{s,t}\leq \overline{P}_{\rm ht}x_{{\rm ht},i},\quad \forall i \in \Lambda,  \forall t \label{cons3-21}\\
&&\sum_{i\in \Lambda_z} gl_{i}^{s,t}+ul_{z}^{s,t}= \xi_{z}^{s,t},\quad \forall z \in \mathcal{Z},  \forall t \label{cons3-22}\\
&&0\leq gl^{s,t}_{i},ul_{z}^{s,t}\leq \xi_{z}^{s,t} ,\quad \forall i \in \Lambda_z,z\in\mathcal{Z},  \forall t \label{cons3-22-1}\\
&&0\leq gl_{i}^{s,t}\leq SR\cdot n_{i} ,\quad \forall i \in \Lambda,  \forall t \label{consHD}
\end{eqnarray}

\textbf{$\bullet$} \textit{Constraints for PDN:}
The operation of PDN is described by the linearized Disflow model, as in \eqref{cons3-25}--\eqref{cons3-35} \cite{DONG2023120849}. By integrating multiple hydrogen-electrical microgrids, the nodal power balance satisfies \eqref{cons3-25}--\eqref{cons3-28}. \eqref{cons3-28-1} is the capacity limit of distribution lines. \eqref{cons3-29} provides a bridge between the nodal voltage and distributed power, which is in the spirit of Ohm's Law \cite{6507355}. The security range of voltage deviation is constrained by \eqref{cons3-30}. In \eqref{cons3-35}, the exchange power through the middle-voltage substation is limited by the transactive capacity. Similar to \eqref{cons3-13}, the quadratic constraints \eqref{cons3-28-1} and \eqref{cons3-35} can also be linearized.
\setlength{\arraycolsep}{-0.5em}
\begin{eqnarray}
&&fp_{ji}^{s,t}-\sum\limits_{l\in\mathcal{C}(i)}fp_{il}^{s,t}= \left\{\begin{array}{l} ~ p_{{\rm lv},i}^{s,t},\; \forall i \in \Lambda,  \forall t \\ ~ pl_{i}^{s,t},\; \forall i \in \mathcal{N^{+}} \backslash \Lambda, \forall t\end{array}\right.\label{cons3-25}\\
&&fq_{ji}^{s,t}-\sum\limits_{l\in\mathcal{C}(i)}fq_{il}^{s,t}= \left\{\begin{array}{l}~ q_{{\rm lv},i}^{s,t},\; \forall i \in \Lambda,  \forall t \\ ~pl_{i}^{s,t}\tan\varphi_{{\rm pl},i}^{s,t}, \forall i \in \mathcal{N^{+}} \backslash \Lambda,  \forall t\end{array}\right.\label{cons3-26}\\
&&\sum\limits_{l\in\mathcal{C}(0)}fp_{0l}^{s,t}=p_{\rm mv}^{s,t}\geq 0,\quad  \forall t \label{cons3-27}\\
&&\sum\limits_{l\in\mathcal{C}(0)}fq_{0l}^{s,t}=q_{\rm mv}^{s,t}\geq 0,\quad  \forall t \label{cons3-28}\\
&&\left(fp_{ij}^{s,t}\right)^2+\left(fq_{ij}^{s,t}\right)^2\leq\left(S_{ij}^{s,t}\right)^2,\quad\forall(i,j)\in\mathcal{L},\forall t\label{cons3-28-1}\\
&&U_i^{s,t}- U_j^{s,t}=\left(r_{ij}fp_{ij}^{s,t}+ x_{ij}fq_{ij}^{s,t}\right)/U_0,\;\forall (i,j)\in\mathcal{L},   \forall t \label{cons3-29}\\
&&U_i^{\rm{min}}\leq U_i^{s,t}\leq U_i^{\rm{max}},\quad \forall i\in\mathcal{N},   \forall t \label{cons3-30}\\
&&\left(p_{\rm mv}^{s,t}\right)^2+\left(q_{\rm mv}^{s,t}\right)^2\leq \left(S_{\rm mv}^{\rm{max}}\right)^2,\quad  \forall t \label{cons3-35}
\end{eqnarray}

Collecting the $|\mathcal{S}|$ branches together, the lower-level problem holds the form of $\sum_{s\in\mathcal{S}}\pi_s\min Q_s(x_{k,i},n_{i},\xi_{z}^{s,t})$ in \eqref{Objective Function:1}, subjecting to constraints \eqref{cons3-1}--\eqref{cons3-35} for all $s\in\mathcal{S}$ as in \eqref{branches}, where $\pi_s$ is the probability of scenario $s$ satisfying $\sum_{s\in\mathcal{S}}\pi_s=1$.

\subsubsection{Middle-Level Problem}
It is essential to further ascertain whether the siting and sizing decisions are profitable enough with the uncertainty in refueling demand, as hydrogen refueling serves as the primary revenue source. Hence, the middle-level problem focuses on identifying the ``worst'' situation of refueling demand, and is represented as the ``$\max$''s in \eqref{Objective Function:1}. Specifically, for each scenario $s$, we adopt a DDU set, as shown in \eqref{ddu-set-1}--\eqref{ddu-set-4}, to model the refueling demand ($\xi_{z}^{s,t}$). 
\setlength{\arraycolsep}{-0.4em}
\begin{eqnarray}
	&&\Xi^s(n_{i})=\Bigg\{\xi_z^{s,t}\geq0~\bigg|~\xi_{z}^{s,t}\leq\overline{\xi}_{z}^{s,t}+\overline{\gamma}_{z}^t\displaystyle\sum_{i\in\Lambda_z} n_{i},\quad \forall z ,\forall t;\label{ddu-set-1}\\
	&&\;\quad\quad\quad\quad\quad\quad\quad\qquad\xi_{z}^{s,t}\geq\underline{\xi}_{z}^{s,t}+\underline{\gamma}^t_{z}\displaystyle\sum_{i\in\Lambda_z} n_{i},\quad \forall z,\forall t;\label{ddu-set-2}\\
	&&\;\quad\quad\quad\quad\quad\quad\quad\qquad\sum\limits_{z\in\mathcal{Z}}\xi_{z}^{s,t}\leq \overline{\zeta}^{s,t}+\overline{\alpha}^{t}\sum_{i\in\Lambda}n_{i},\quad \forall t;\label{ddu-set-3}\\
	&&\;\quad\quad\quad\quad\quad\quad\quad\qquad\sum\limits_{z\in\mathcal{Z}}\xi_{z}^{s,t}\geq\underline{\zeta}^{s,t}+\underline{\alpha}^{t}\sum_{i\in\Lambda}n_{i},\quad \forall t \Bigg\} \label{ddu-set-4}
\end{eqnarray}
\eqref{ddu-set-1} and \eqref{ddu-set-2} respectively define the upper and lower bounds on intra-zonal refueling demand $\xi_{z}^{s,t}$, where $\overline{\xi}_{z}^{s,t}$ and $\underline{\xi}_{z}^{s,t}$ represent the basic bounds, and induced coefficients $\overline{\gamma}_{z}^t$ and $\underline{\gamma}_{z}^t$ are introduced to reflect DIE in each zone. The number of HDs, whose information is available to HV drivers through the transportation information system, is chosen as the representative of microgrid refueling capacities. The DDU set establishes a relationship between refueling demand and the number of HDs, rendering the bounds of $\xi_{z}^{s,t}$ increase with $\sum_{i\in\Lambda_z} n_{i}$. Shift the perspective from a specific zone $z\in\mathcal{Z}$ to encompassing the entire town, \eqref{ddu-set-3} and \eqref{ddu-set-4} provide the variation interval of total refueling demand $\sum_{z\in\mathcal{Z}}\xi_{z}^{s,t}$, which is also decision-dependent. The meanings of $\overline{\zeta}^{s,t}/\underline{\zeta}^{s,t}$  (resp. $\overline{\alpha}^{t}/\underline{\alpha}^{t}$) are analogous to those of $\overline{\xi}_{z}^{s,t}/\underline{\xi}_{z}^{s,t}$  (resp. $\overline{\gamma}_{z}^t/\underline{\gamma}_{z}^t$). Generally, we have $[\underline{\zeta}^{s,t}+\underline{\alpha}^{t}\sum_{i\in\Lambda}n_{i}, \overline{\zeta}^{s,t}+\overline{\alpha}^{t}\sum_{i\in\Lambda}n_{i}]\subseteq[\sum_{z\in\mathcal{Z}}( \underline{\xi}_{z}^{s,t}+\underline{\gamma}^t_{z}\sum_{i\in\Lambda_z} n_{i}) ,\sum_{z\in\mathcal{Z}}(\overline{\xi}_{z}^{s,t}+\overline{\gamma}_{z}^t\sum_{i\in\Lambda_z} n_{i}) ]$. Hence, mathematically, \eqref{ddu-set-3} and \eqref{ddu-set-4} also serve to avoid over conservativeness of the DDU sets.

Fig. \ref{ddu-set-ill} provides an example for illustration of the DDU and DIU sets. Feasible regions of DDU and DIU sets are built by solid and dashed lines, respectively. The DIU set overlooks the DIE, and can be obtained by omitting the induced terms in $\Xi^s(n_{i})$, i.e., setting $\overline{\gamma}_{z}^t=\underline{\gamma}_{z}^t=\overline{\alpha}^t=\underline{\alpha}^t=0$. It shows that, due to DIE, all bounds of the DDU set see an increase with the number of HDs ($n_i$) compared to those of the DIU one. Hence, if the DIU set is adopted for NHEMP decisions, the derived solution may seriously underperform with respect to the investment based on the DDU set, since the former one fails to take advantage of the profitability in hydrogen provision associated with the induced refueling demand.

\begin{remark}	
	When $\overline{\gamma}_{z}^t=\underline{\gamma}_{z}^t=\overline{\alpha}^t=\underline{\alpha}^t=0$, $\Xi^s(n_{i})$ will reduce to the static DIU set. Hence, the traditional DIU model is a special case of the proposed DDU one. In this regard, the properties and solution method presented in this paper are also applicable to the static trilevel DIU problem.
\end{remark}

\begin{figure}[]
	\centering
	\includegraphics[scale=0.42]{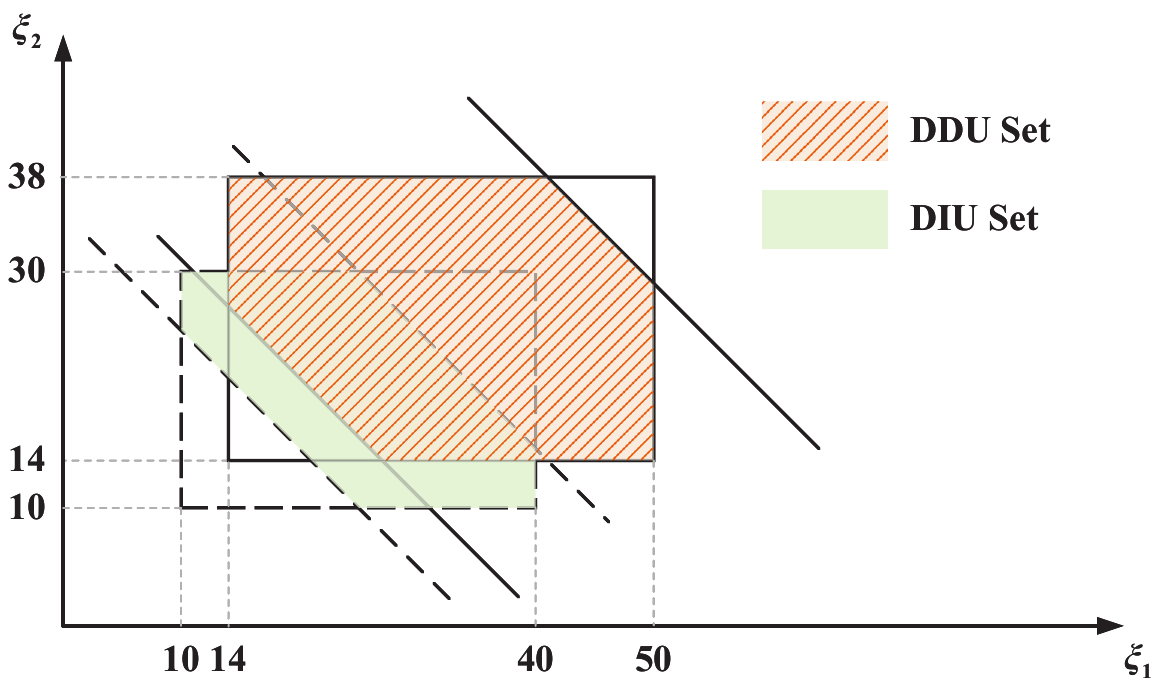}
	\caption{Illustration of the DDU and DIU sets. (Without loss of generality, indices $s$ and $t$ are suppressed, and $\left|\mathcal{Z} \right|=2$. Then, $\underline{\xi}_{1}$, $\overline{\xi}_{1}$, $\underline{\xi}_{2}$, $\overline{\xi}_{2}$ and $\underline{\zeta}$, $\overline{\zeta}$ are set to 10, 40, 10, 30, 35 and 55, respectively. Besides, $\sum_{i\in\Lambda_1}n_i=2$ and $\sum_{i\in\Lambda_2}n_i=1$. For the DDU set, $\underline{\gamma}_{1}$, $\overline{\gamma}_{1}$, $\underline{\gamma}_{2}$, $\overline{\gamma}_{2}$, $\underline{\alpha}$ and $\overline{\alpha}$ are, respectively, chosen as 2, 5, 4, 8, 2 and 8. For the DIU one, $\overline{\gamma}_{z}=\underline{\gamma}_{z}=\overline{\alpha}=\underline{\alpha}=0$.)}
	\label{ddu-set-ill}
\end{figure}

\subsection{Compact Formulation}
For the sake of compactness, and without any loss of generality, we mathematically express the NHEMP problem in \eqref{Objective Function:1}--\eqref{ddu-set-4} in the matrix form as \textbf{DDU-NHEMP} below:
\begin{eqnarray}
	\nonumber&&\bm{{\rm DDU-NHEMP}}:\\
	&&\Phi~=~\min_{\bm{x}\in\mathcal{X}}~\bm{c}^{\intercal} \bm{x}+\sum_{s\in\mathcal{S}}\pi_s\max_{\bm{\xi}_s\in\Xi_s(\bm{x})}\min_{\bm{y}_s\in\mathcal{Y}_s(\bm{x},\bm{\xi}_s)}\bm{d}_s^{\intercal}\bm{y}_s,\label{RO-DDU}
\end{eqnarray}
where
\begin{eqnarray}
	&&\mathcal{X}=\left\{\bm{x}\in\left\lbrace0,1 \right\rbrace^{n_{\rm x_1}}\times \mathbb{N}^{n_{\rm x_2}}\;\big|\;\bm{Ax}\leq\bm{b}\right\},\\
	&&\Xi_s(\bm{x})=\left\{\bm{ \xi}_s\in\mathbb{R}_+^{n_{\rm \xi}}\;\big|\;\bm{H}_s\bm{\xi}_s\leq\bm{h}_s-\bm{F}_s\bm{x}\right\},\label{set-Xi}\\
	&&\mathcal{Y}_s(\bm{x},\bm{\xi}_s)=\left\{\bm{y}_s\in\mathbb{R}_+^{n_{\rm y}}\;\big|\;\bm{B}_s\bm{y}_s\geq\bm{f}_s -\bm{G}_s\bm{x}-\bm{E}_s\bm{\xi}_s\right\}.\label{set-y}
\end{eqnarray}

In \textbf{DDU-NHEMP},  $\bm{x}$  denotes the upper-level decision vector for hydrogen-electrical microgrids' investment strategies, including $u_i$, $x_{k,i}$, and $n_i$. $\pi_s$ is the probability of scenario $s\in\mathcal{S}$. Vector $\bm{\xi}_s$ expresses the associated DDU parameters,  $\xi_z^{s,t}$, in the middle-level problem with scenario $s$. $\bm{y}_s$ represents the operation-related decision vector of branch $s$ in the lower-level, including $p_{\rm mv}^{s,t}$, $g_{{\rm pur},i}^{s,t}$, $ls_i^{s,t}$, $pc_{k,i}^{s,t}$, $pd_{k,i}^{s,t}$, $pl_i^{s,t}$, $gl_i^{s,t}$, $p_{k,i}^{s,t}$, $q_{k,i}^{s,t}$, $soc_{k,i}^{s,t}$, $p_{{\rm lv},i}^{s,t}$, $p_{{\rm elz},i}^{s,t}$, $q_{{\rm lv},i}^{s,t}$, $g_{{\rm elz},i}^{s,t}$, $loh_{{\rm ht},i}^{s,t}$, $ul_z^{s,t}$, $fp_{ij}^{s,t}$, $fq_{ij}^{s,t}$, $q_{\rm mv}^{s,t}$, and $U_i^{s,t}$.  $n_{\rm x_1}/n_{\rm x_2}$, $n_{\rm \xi}$, and $n_{\rm y}$ are appropriate quantities standing for dimensions of $\bm{x}$, $\bm{y}_s$, and $\bm{\xi}_s$, respectively. Coefficient vectors $\bm{c}$, $\bm{d}_s$, $\bm{b}$, $\bm{h}_s$, and $\bm{f}_s$, and constraint matrices $\bm{A}$,  $\bm{H}_s$,  $\bm{F}_s$, $\bm{B}_s$,  $\bm{G}_s$, and $\bm{E}_s$ are all with compatible dimensions. We wish to point out that vectors in this paper are all column vectors. 

The objective function is abstracted from \eqref{Objective Function:1} and \eqref{Objective Function:3}. Specifically, $\bm{c}^{\intercal}\bm{x}$ represents the upper-level objective function as in Line 1 of \eqref{Objective Function:1}, and $\bm{d}_s^{\intercal}\bm{y}_s$ is for the $s$-th branch of the lower-level problem as in \eqref{Objective Function:3}. Defined by constraints \eqref{cons2-1}--\eqref{cons2-2-1}, $\mathcal{X}$ indicates the feasible set of $\bm{x}$. $\Xi_s(\bm{x})$ is the DDU set as exhibited in \eqref{ddu-set-1}--\eqref{ddu-set-4}, which is dependent on $\bm{x}$. Further, corresponding to \eqref{cons3-1}--\eqref{cons3-35}, $\mathcal{Y}_s(\bm{x},\bm{\xi}_s)$ shows the feasible region of $\bm{y}_s$. Moreover, we use $m_{\rm \xi}$ and $m_{\rm y}$ to denote the numbers of constraints (rows) of $\Xi_s(\bm{x})$ and $\mathcal{Y}_s(\bm{x},\bm{\xi}_s)$, respectively.

\subsection{Some Structural Properties}\label{analysis}

\subsubsection{Relatively Complete Recourse} \label{complete recourse}
\textbf{DDU-NHEMP} possesses the \textit{relatively complete recourse}, i.e., for $s\in\mathcal{S}$, given any $\hat{\bm{x}}\in\mathcal{X}$ and $\hat{\bm{\xi}}_s\in\Xi_s(\hat{\bm{x}})$, there exists feasible $\bm{y}_s$ such that $\bm{y}_s\in\mathcal{Y}_s(\hat{\bm{x}},\hat{\bm{\xi}}_s)$, and the optimal value of $\min_{\bm{y}_s\in\mathcal{Y}_s(\hat{\bm{x}},\hat{\bm{\xi}}_s)}\bm{d}_s^{\intercal}\bm{y}_s$ is finite. This is achieved due to the inclusion of non-negative auxiliary variables $gl_{i}^{s,t}$ and $pl_i^{s,t}$.

\subsubsection{Definition of $\mathcal{OU}_s$}\label{def-OU}
For a given $\bm{x}\in\mathcal{X}$, the $\max_{\bm{\xi}_s\in\Xi_s(\bm{x})}\min_{\bm{y}_s\in\mathcal{Y}_s(\bm{x},\bm{\xi}_s)}\bm{d}_s^{\intercal}\bm{y}_s$ corresponding to each scenario $s\in\mathcal{S}$  can be converted to the following single-level problem through the duality theory \cite{karloff2008linear}:
\begin{eqnarray}
	&&\max_{\bm{\xi}_s\in\Xi_s(\bm{x}),\bm{\lambda}_s\in\Pi_s}\left(\bm{f}_s-\bm{G}_s\bm{x}-\bm{E}_s\bm{\xi}_s\right)^{\intercal}\bm{\lambda}_s\label{dual-2stage},\quad \forall s\in\mathcal{S},
\end{eqnarray}
where $\bm{\lambda}_s\in\mathbb{R}_{+}^{m_{\rm y}}$ is the multiplier of \eqref{set-y} and its feasible set $\Pi_s$ is defined as:
\begin{eqnarray}
	&&\Pi_s=\left\lbrace \bm{\lambda}_s\in\mathbb{R}_{+}^{m_{\rm y}}\;\big|\;\bm{d}_s-\bm{B}_s^{\intercal}\bm{\lambda}_s\geq\bm{0}\right\rbrace,\quad \forall s\in\mathcal{S}.
\end{eqnarray}
Of note, problem \eqref{dual-2stage} is a bilinear program. By alternately fixing $\bm{\xi}_s$ and $\bm{\lambda}_s$, and leveraging the property of linear programs (LPs) \cite{karloff2008linear}, a non-trivial proposition can be obtained:

\begin{prop}\label{prop-1}
For given $\bm{x}\in\mathcal{X}$ and $s\in\mathcal{S}$, when \eqref{dual-2stage} reaches its optimum, there exist optimal $\hat{\bm{\xi}}_s$ and $\hat{\bm{\lambda}}_s$ being one of the extreme points of $\Xi_s(\bm{x})$ and $\Pi_s$, respectively. And $\hat{\bm{\xi}}_s$ is also an optimal solution to $\max_{\bm{\xi}_s\in\Xi_s(\bm{x})}\min_{\bm{y}_s\in\mathcal{Y}_s(\bm{x},\bm{\xi}_s)}\bm{d}_s^{\intercal}\bm{y}_s$.
\end{prop}

Different from $\Xi_s(\bm{x})$, $\Pi_s$ is a fixed polyhedron that is independent of $\bm{x}$. With a given $\bm{\lambda}_s\in\Pi_s$, the optimal $\bm{\xi}_s$ of problem \eqref{dual-2stage} can be obtained by solving the following LP:
\begin{eqnarray}
	&&\bm{\Psi}_s(\bm{x},\bm{\lambda}_s):\max_{\bm{\xi}_s\in\Xi_s(\bm{x})}\left(-\bm{E}_s\bm{\xi}_s\right)^{\intercal}\bm{\lambda}_s,\quad\forall s\in\mathcal{S}.
\end{eqnarray}
By resorting to its Karush–Kuhn–Tucker (KKT) conditions \cite{boyd2004convex}, the solution
set of $\bm{\Psi}_s(\bm{x},\bm{\lambda}_s)$ can be defined as:
\begin{eqnarray}
	&&\mathcal{OU}_s(\bm{x},\bm{\lambda}_s)\!=\!\Big\{\bm{\xi}_s\in\mathbb{R}_+^{n_{\rm \xi}},\bm{\vartheta}_s\in\mathbb{R}_{+}^{m_{\rm \xi}}\Big|\bm{H}_s\bm{\xi}_s\leq\bm{h}_s-\bm{F}_s\bm{x},\label{ou-1}\\
	&&\qquad\qquad\;\;\;\;\quad\quad\bm{H}_s^{\intercal}\bm{\vartheta}_s+\bm{E}_s^{\intercal}\bm{\lambda}_s\geq\bm{0},\label{ou-2}\\
	&&\qquad\qquad\;\;\;\;\quad\quad\bm{\vartheta}_s\circ(\bm{h}_s-\bm{F}_s\bm{x}-\bm{H}_s\bm{\xi}_s)=\bm{0},\label{ou-3}\\
	&&\qquad\qquad\;\;\;\;\quad\quad\bm{\xi}_s\circ(\bm{H}_s^{\intercal}\bm{\vartheta}_s+\bm{E}_s^{\intercal}\bm{\lambda}_s)=\bm{0}\Big\},\label{ou-4}
\end{eqnarray}
where $\bm{\vartheta}_s\in\mathbb{R}_{+}^{m_{\rm \xi}}$ is the corresponding multiplier of \eqref{set-Xi}, and $\circ$ denotes the Hadamard product. In $\mathcal{OU}_s(\bm{x},\bm{\lambda}_s)$, \eqref{ou-1} and \eqref{ou-2} are, respectively, primal and dual constraints of $\bm{\Psi}_s(\bm{x},\bm{\lambda}_s)$, and \eqref{ou-3} and \eqref{ou-4} are complementarity constraints. We wish to point out that $\mathcal{OU}_s$ will serve as an important component in our solution algorithm.\\

We note that the stochastic-robust \textbf{DDU-NHEMP} problem is computationally very challenging. On the one hand, it holds a complicated $\min-\max-\min$ structure with DDU sets. The traditional C\&CG method \cite{zeng2013solving}, a recognized algorithm developed for this type of problems with DIU sets, cannot handle it, since the DDU set $\Xi_s(\bm{x})$ varies with the upper-level decision $\bm{x}$ dynamically. In the DDU setting, the ``worst'' extreme point of $\Xi_s(\bm{x}^1)$ derived for $\bm x^1$ will generally not be an extreme point of $\Xi_s(\bm{x}^2)$ when $\bm x^1\neq \bm x^2$, or even may not belong to $\Xi_s(\bm{x}^2)$, which destroys the traditional C\&CG procedure. On the other hand, the \textbf{DDU-NHEMP} problem integrates a couple of stochastic scenarios in the lower-level. The computational complexity grows significantly with the number of scenarios. Fortunately, the scenario reduction strategy has been presented and applied in the literature to balance the computational efficiency and accuracy, e.g., in Refs. \cite{dupavcova2003scenario,heitsch2003scenario}.  This strategy is also adopted in our work to generate a typical scenario set to deal with this issue. Hence, in this paper, we only focus on developing an adaptive algorithm to efficiently address the challenging trilevel DDU problem with typical scenarios.

\section{Parametric Column-and-Constraint Generation Algorithm}\label{stochastic}
In this section, the PC\&CG algorithm, an advanced and exact decomposition method designed for DDU problems with the $\min-\max-\min$ structure, is introduced to address  \textbf{DDU-NHEMP}. We customize and develop the original PC\&CG in Ref. \cite{zeng2022two} to accommodate multiple DDU sets, which are defined with respect to $|\mathcal{S}|$ scenarios.

\subsection{Equivalent Single-Level Reformulation}\label{sec:single-level}
For each scenario $s\in\mathcal{S}$, let $\mathcal{P}_s$ be the set of extreme points of $\Pi_s$. Note that $\mathcal{P}_s$ is also independent of $\bm{x}$ because $\mathcal{P}_s\subseteq\Pi_s$. In Theorem \ref{thm-1} below, we present an equivalent single-level reformulation of \textbf{DDU-NHEMP} (referred to as \textbf{DDU-Single}). It is worth highlighting that this reformulation achieves the reduction in problem levels, and establishes the foundation of our solution method.

\begin{thm}\label{thm-1}
	\textbf{DDU-NHEMP} is equivalent to the single-level mixed-integer program (MIP):
\begin{align}
	\bm{{\rm DD}}&\bm{{\rm U-Single}}:~
	\min~\bm{c}^{\intercal} \bm{x}+\sum_{s\in\mathcal{S}}\pi_s\eta_s\label{Single}\\
	{\rm s.t.}~&A\bm{x}\leq\bm{b}\\
	&\eta_s\geq \bm{d}_s^{\intercal}\bm{y}_s^{\bm{\lambda}_s},\quad \forall\bm{\lambda}_s\in\mathcal{P}_s,\forall s\in\mathcal{S}\\
	&\bm{B}_s\bm{y}_s^{\bm{\lambda}_s}\geq\bm{f}_s -\bm{G}_s\bm{x}-\bm{E}_s\bm{\xi}_s^{\bm{\lambda}_s},\quad \forall\bm{\lambda}_s\in\mathcal{P}_s,\forall s\in\mathcal{S}\\
	&(\bm{\xi}_s^{\bm{\lambda}_s},\bm{\vartheta}_s^{\bm{\lambda}_s})\in\mathcal{OU}_s(\bm{x},\bm{\lambda}_s),\quad \forall\bm{\lambda}_s\in\mathcal{P}_s,\forall s\in\mathcal{S}\\
	&\bm{x}\in\left\lbrace0,1 \right\rbrace^{n_{\rm x_1}}\times \mathbb{N}^{n_{\rm x_2}},\bm{y}_s^{\bm{\lambda}_s}\in\mathbb{R}^{n_{\rm y}}_+,
\end{align}
where superscript $\bm{\lambda}_s$ is used to make a distinction between variables corresponding to different extreme points of $\Pi_s$.
\end{thm}

\begin{remark}
	The equivalence of \textbf{DDU-NHEMP} and \textbf{DDU-Single} implies that these two problems share the same optimal objective value, and the optimal upper-level solution of \textbf{DDU-NHEMP} is also optimal to \textbf{DDU-Single} and vice versa.
\end{remark}

\begin{proof}
To support our proof of Theorem \ref{thm-1}, the following lemma is introduced, whose detailed proof can be found in Appendix \ref{appen-1}.
\begin{lem}\label{lemma-1}
	For a given $\bm{x}$ and for each $s\in\mathcal{S}$,
	\begin{eqnarray}
		&&\max_{\bm{\xi}_s\in\Xi_s(\bm{x})}\min_{\bm{y}_s\in\mathcal{Y}_s(\bm{x},\bm{\xi}_s)}\bm{d}_s^{\intercal}\bm{y}_s=\max_{\bm{\xi}_s\in\Xi_s^*(\bm{x})}\min_{\bm{y}_s\in\mathcal{Y}_s(\bm{x},\bm{\xi}_s)}\bm{d}_s^{\intercal}\bm{y}_s,
	\end{eqnarray}
	where $\Xi_s^*(\bm{x})=\textstyle\bigcup_{\bm{\lambda}_s\in\mathcal{P}_s}\mathcal{OU}_s^{\bm{\xi}_s}(\bm{x},\bm{\lambda}_s)$, and for each $\bm{\lambda}_s\in\mathcal{P}_s$, $\mathcal{OU}_s^{\bm{\xi}_s}(\bm{x},\bm{\lambda}_s)$ is the projection of $\mathcal{OU}_s(\bm{x},\bm{\lambda}_s)$ onto its subspace hosting $\bm{\xi}_s$, i.e., $\mathcal{OU}_s^{\bm{\xi}_s}(\bm{x},\bm{\lambda}_s)\triangleq \{\bm{\xi}_s\;|\;(\bm{\xi}_s,\bm{\vartheta}_s)\in\mathcal{OU}_s(\bm{x},\bm{\lambda}_s)~{\rm for~ some}~\bm{\vartheta}_s\}$.
\end{lem}

\noindent\textit{Proof of Theorem \ref{thm-1}.}
By applying Lemma \ref{lemma-1}, \textbf{DDU-NHEMP} is equivalent to
\begin{align}
	&\min_{\bm{x}\in\mathcal{X}}~\bm{c}^{\intercal} \bm{x}+\sum_{s\in\mathcal{S}}\pi_s\max_{\bm{\xi}_s\in\Xi_s^*(\bm{x})}\min_{\bm{y}_s\in\mathcal{Y}_s(\bm{x},\bm{\xi}_s)}\bm{d}_s^{\intercal}\bm{y}_s \label{pf-1-1}.
\end{align}
Since $\Xi_s^*(\bm{x})$ is a union set of $\mathcal{OU}_s^{\bm{\xi}_s}(\bm{x},\bm{\lambda}_s)$'s according to each $\bm{\lambda}_s\in\mathcal{P}_s$, by enumerating the $\bm{\lambda}_s$'s of $\mathcal{P}_s$ for every $s\in\mathcal{S}$, problem \eqref{pf-1-1} can be converted to
\begin{align}
	\min~&\bm{c}^{\intercal} \bm{x}+\sum_{s\in\mathcal{S}}\pi_s\eta_s \label{pf-1-2}\\
	{\rm s.t.}~&\bm{x}\in\mathcal{X}\label{pf-1-3}\\
	&\nonumber\eta_s\geq \Big\{ \bm{d}_s^{\intercal}\bm{y}_s^{\bm{\lambda}_s}\;|\;\bm{y}_s^{\bm{\lambda}_s}\in\mathcal{Y}_s(\bm{x},\bm{\xi}_s^{\bm{\lambda}_s}),\\
	&\qquad\quad\bm{\xi}_s^{\bm{\lambda}_s}\in\mathcal{OU}_s^{\bm{\xi}_s}(\bm{x},\bm{\lambda}_s)\Big\} ,\; \forall \bm{\lambda}_s\in\mathcal{P}_s,\forall s\in\mathcal{S}\label{pf-1-4}.
\end{align}
Considering the projection relationship between $\mathcal{OU}_s^{\bm{\xi}_s}(\bm{x},\bm{\lambda}_s)$ and $\mathcal{OU}_s(\bm{x},\bm{\lambda}_s)$, \textbf{DDU-Single} can be readily obtained from problem \eqref{pf-1-2}--\eqref{pf-1-4}.
\end{proof}

In \textbf{DDU-Single}, we enumerate all extreme points of $\Pi_s$'s for each scenario $s$. However, the total number of the elements of $\mathcal{P}_s$'s, i.e., $\sum_{s\in\mathcal{S}}\left| \mathcal{P}_s\right|$, could be very large, rendering it impractical to solve \textbf{DDU-Single} directly. Hence, in the next subsection, we customize and develope the PC\&CG algorithm to achieve an exact and high-efficient solution to \textbf{DDU-Single}. 

\subsection{PC\&CG Decomposition}
PC\&CG decomposition holds a master-subproblem architecture. It constructs and conducts computation on a relaxation of \textbf{DDU-Single}, referred to as the master problem, following by iteratively generating valuable variables and constraints through a series of subproblems to strengthen this relaxation, until a convergence condition is met.

Specifically, in each iteration, PC\&CG begins with the master problem (\textbf{MP}) as in \eqref{MP-1}--\eqref{MP-8}. In constraints \eqref{MP-3}--\eqref{MP-5}, for each scenario $s$, $\hat{\mathcal{P}}_s$ is a subset of $\mathcal{P}_s$, i.e., $\hat{\mathcal{P}}_s\subseteq\mathcal{P}_s$. Built on $\hat{\mathcal{P}}_s$'s, \textbf{MP} provides a relaxation of \textbf{DDU-Single}.
\begin{align}
	\bm{{\rm MP}}&:~\min~\bm{c}^{\intercal} \bm{x}+\sum_{s\in\mathcal{S}}\pi_s\eta_s\label{MP-1}\\
	{\rm s.t.}~&A\bm{x}\leq\bm{b}\label{MP-2}\\
	&\eta_s\geq \bm{d}_s^{\intercal}\bm{y}_s^{\bm{\lambda}_s},\quad \forall\bm{\lambda}_s\in\hat{\mathcal{P}}_s,\forall s\in\mathcal{S}\label{MP-3}\\
	&\bm{B}_s\bm{y}_s^{\bm{\lambda}_s}\geq\bm{f}_s -\bm{G}_s\bm{x}-\bm{E}_s\bm{\xi}^{\bm{\lambda}_s},\quad \forall\bm{\lambda}_s\in\hat{\mathcal{P}}_s,\forall s\in\mathcal{S}\label{MP-4}\\
	&(\bm{\xi}_s^{\bm{\lambda}_s},\bm{\vartheta}_s^{\bm{\lambda}_s})\in\mathcal{OU}_s(\bm{x},\bm{\lambda}_s),\quad \forall\bm{\lambda}_s\in\hat{\mathcal{P}}_s,\forall s\in\mathcal{S}\label{MP-5}\\
	&\bm{x}\in\left\lbrace0,1 \right\rbrace^{n_{\rm x_1}}\times \mathbb{N}^{n_{\rm x_2}},\bm{y}_s^{\bm{\lambda}_s}\in\mathbb{R}^{n_{\rm y}}_+\label{MP-8}
\end{align}
Collecting the optimal $(\hat{\bm{x}},\hat{\eta}_1,\cdots,\hat{\eta}_{\left|\mathcal{S}\right|})$ from \textbf{MP}, a lower bound (LB) of \textbf{DDU-Single} can be obtained as:
\begin{align}
	LB=\bm{c}^{\intercal}\hat{\bm{x}}+\sum_{s\in\mathcal{S}}\pi_s\hat{\eta}_s.
\end{align}

With the derived upper-level solution $\hat{\bm{x}}$ from \textbf{MP}, in order to strengthen \textbf{MP}, $\left| \mathcal{S}\right| $ subproblems (\textbf{SP}$_s$'s) are defined as:
\begin{align}
	\bm{{\rm SP}}_s:\;
\max_{\bm{\xi}_s\in\Xi_s(\hat{\bm{x}})}\min_{\substack{\bm{y}_s\in\mathcal{Y}_s(\hat{\bm{x}},\bm{\xi}_s)}}\bm{d}_s^{\intercal}\bm{y}_s, \quad s\in\mathcal{S}.\label{SP2}
\end{align}
Note that \textbf{SP}$_s$ is always feasible due to the relatively complete recourse property as indicated in Section \ref{complete recourse}. Additionally, we mention that \textbf{SP}$_s$ is a bilevel LP in the $\max-\min$ form. By employing the KKT conditions of its lower-level minimization problem, \textbf{SP}$_s$ can be converted to a single-level problem as follows:
\begin{align}
	\max~&\bm{d}_s^{\intercal}\bm{y}_s\label{kkt-1}\\
	{\rm s.t.}~&\bm{H}_s\bm{\xi}_s\leq\bm{h}_s-\bm{F}_s\hat{\bm{x}}\label{kkt-2}\\
	&\bm{B}_s\bm{y}_s+\bm{E}_s\bm{\xi}_s\geq\bm{f}_s -\bm{G}_s\hat{\bm{x}}\label{kkt-3}\\
	&\bm{d}_s-\bm{B}_s^{\intercal}\bm{\lambda}_s\geq\bm{0}\label{kkt-4}\\
	&\bm{\lambda}_s\circ(\bm{B}_s\bm{y}_s+\bm{E}_s\bm{\xi}_s-\bm{f}_s +\bm{G}_s\hat{\bm{x}})=\bm{0}\label{kkt-5}\\
	&\bm{y}_s\circ (\bm{d}_s-\bm{B}_s^{\intercal}\bm{\lambda}_s)=\bm{0}\label{kkt-5-1}\\
	&\bm{\xi}_s\in\mathbb{R}_+^{n_{\rm \xi}},\bm{y}_s\in\mathbb{R}_+^{n_{\rm y}},\bm{\lambda}_s\in\mathbb{R}_+^{m_{\rm y}}\label{kkt-6},
\end{align}
where \eqref{kkt-3} and \eqref{kkt-4} are, respectively, primal and dual constraints of the lower-level problem of \textbf{SP}$_s$, and \eqref{kkt-5} and \eqref{kkt-5-1} are complementarity constraints.

By solving \textbf{SP}$_s$'s, the optimal lower-level variables and extreme points of $\Pi_s$'s, i.e.,  $(\hat{\bm{y}}_1,\hat{\bm{\lambda}}_1,\cdots,\hat{\bm{y}}_{\left|\mathcal{S}\right|},\hat{\bm{\lambda}}_{\left|\mathcal{S}\right|} )$, can be derived. Then, for each $s$, we will update $\hat{\mathcal{P}}_s\leftarrow\hat{\mathcal{P}}_s\cup\{\hat{\bm{\lambda}}_s\}$, create new variables $\bm{y}_s^{\hat{\bm{\lambda}}_s}$, $\bm{\xi}_s^{\hat{\bm{\lambda}}_s}$ and $\bm{\vartheta}_s^{\hat{\bm{\lambda}}_s}$, and add new constraints  \eqref{Add-1}--\eqref{Add-3} as an \textit{optimality cut} to \textbf{MP}, so as to strengthen the \textbf{MP}.
\begin{align}
	&\eta_s\geq \bm{d}_s^{\intercal}\bm{y}_s^{\hat{\bm{\lambda}}_s},\bm{y}_s^{\hat{\bm{\lambda}}_s}\in\mathbb{R}^{n_{\rm y}}_+,\quad s\in\mathcal{S}\label{Add-1}\\
	&\bm{B}_s\bm{y}_s^{\hat{\bm{\lambda}}_s}\geq\bm{f}_s -\bm{G}_s\bm{x}-\bm{E}_s\bm{\xi}_s^{\hat{\bm{\lambda}}_s},\quad s\in\mathcal{S}\label{Add-2}\\
	&(\bm{\xi}_s^{\hat{\bm{\lambda}}_s},\bm{\vartheta}_s^{\hat{\bm{\lambda}}_s})\in\mathcal{OU}_s(\bm{x},\hat{\bm{\lambda}}_s),\quad s\in\mathcal{S}\label{Add-3}
\end{align}
Meanwhile, since the $\hat{\bm{x}}$ from \textbf{MP} and the $(\hat{\bm{y}}_1,\cdots,\hat{\bm{y}}_{\left|\mathcal{S}\right|})$ from \textbf{SP}$_s$'s constitute a feasible solution of \textbf{DDU-NHEMP}, the upper bound (UB) of \textbf{DDU-Single} can be updated:
\begin{align}
	&UB=\min\left\lbrace UB,~ \bm{c}^{\intercal}\hat{\bm{x}}+\sum_{s\in\mathcal{S}}\pi_s\bm{d}_s^{\intercal}\hat{\bm{y}}_s\right\rbrace .\label{def-ub}
\end{align}

\begin{remark}
	Complementary slackness conditions \eqref{ou-3} and \eqref{ou-4} in $\mathcal{OU}_s(\bm{x},\bm{\lambda}_s)$ of \textbf{MP} and \eqref{kkt-5} and \eqref{kkt-5-1} of \textbf{SP}$_s$'s exhibit a bilinear structure, which renders difficulty in problem solution. To alleviate the computational burden, we use the so-called ``big-M'' method \cite{ding2016two,6812211} to make the linearization. For example, for the $j$-th constraint of \eqref{kkt-5}, i.e., 
	\begin{eqnarray}
		&&\lambda^j_s(\bm{B}\bm{y}_s+\bm{E}_s\bm{\xi}_s-\bm{f}_s +\bm{G}_s\hat{\bm{x}})^j=0,\label{complement-example}
	\end{eqnarray}
	it can equivalently be replaced by the following linear constraints:
	\begin{eqnarray}
		&&\lambda_s^j\leq Mw_s^j,\label{kkt-7}\\
		&&(\bm{B}\bm{y}_s+\bm{E}_s\bm{\xi}_s-\bm{f}_s +\bm{G}_s\hat{\bm{x}})^j\leq M(1-w_s^j),\label{kkt-8}
	\end{eqnarray}
	where $M$ is a sufficiently large number, e.g., $10^8$, and $w_s^j$ is a binary variable. Taking \eqref{kkt-7} and \eqref{kkt-8} together with \eqref{kkt-3} and the nonnegativity of $\lambda_s^j$ in \eqref{kkt-6}, when $\omega_s^j=0$, we have $\lambda_s^j=0$ and \eqref{kkt-8} is relaxed; when $\omega_s^j=1$, we have $(\bm{B}\bm{y}_s+\bm{E}_s\bm{\xi}_s-\bm{f}_s +\bm{G}_s\hat{\bm{x}})^j=0$ and \eqref{kkt-7} is relaxed. So, the same effect of the complementary constraint \eqref{complement-example} can be achieved by linear constraints \eqref{kkt-7} and \eqref{kkt-8} with binary variable $\omega_s^j$. By adopting such a method to the complementarity constraints, the resulting \textbf{MP} and \textbf{SP}$_s$'s are all mixed-integer linear programs (MILPs), which can be directly solved by the off-the-shelf solvers, e.g., Gurobi or CPLEX.
\end{remark}

Finally, PC\&CG will be terminated if the relative gap ($\frac{UB-LB}{\left| LB\right|}$) is not larger than a predefined tolerance level $\varepsilon$, i.e., $\frac{UB-LB}{\left| LB\right|}\leq\varepsilon$. The overall flow of PC\&CG is described in Algorithm \ref{alg}, and its finite convergence result is established in Theorem \ref{thm-2}.
\begin{algorithm}[]

	\caption{Parametric Column-and-Constraint Generation}

	\begin{algorithmic}[1]

		\STATE {\bf Initialization:} Set $LB\leftarrow-\infty$, $UB\leftarrow+\infty$,  and $\varepsilon=0$; Initial $\hat{\mathcal{P}}_s$'s.

		\WHILE  {$\frac{UB-LB}{\left| LB\right| } > \varepsilon$}

		\STATE Solve master problem ${\bf MP}$.

		\IF {${\bf MP}$ is infeasible}

		\STATE Terminate the algorithm and \emph{report infeasibility}.
		\ELSE

		\STATE Derive the optimal solution $(\hat{\bm{x}},\hat{\eta}_1,\cdots,\hat{\eta}_{\left|\mathcal{S}\right|})$;
		\STATE Update $LB\leftarrow \bm{c}^{\intercal}\hat{\bm{x}}+\sum_{s\in\mathcal{S}}\pi_s\hat{\eta}_s$.
		\ENDIF

		\FOR {$s=1,\cdots,\left|\mathcal{S}\right|$}
		\STATE Solve subproblem ${\bf SP}_s$;
		
		\STATE Obtain the optimal solution $(\hat{\bm{y}}_1,\hat{\bm{\lambda}}_1,\cdots,\hat{\bm{y}}_{\left|\mathcal{S}\right|},\hat{\bm{\lambda}}_{\left|\mathcal{S}\right|} )$;
		\STATE Update $\hat{\mathcal{P}}_s\leftarrow\hat{\mathcal{P}}_s\cup\{\hat{\bm{\lambda}}_s\}$; Create variables $\bm{y}_s^{\hat{\bm{\lambda}}_s}$, $\bm{\xi}_s^{\hat{\bm{\lambda}}_s}$ and $\bm{\vartheta}_s^{\hat{\bm{\lambda}}_s}$; Add constraints \eqref{Add-1}-\eqref{Add-3} to ${\bf MP}$.
		\ENDFOR

		\STATE Update $UB\leftarrow\min\left\lbrace UB,~ \bm{c}^{\intercal}\hat{\bm{x}}+\sum_{s\in\mathcal{S}}\pi_s\bm{d}_s^{\intercal}\hat{\bm{y}}_s\right\rbrace $.

		\ENDWHILE

		\STATE {\bf Output:} Return the solution of \textbf{DDU-NHEMP} $\bm{x}^{\star}=\hat{\bm{x}}$.

	\end{algorithmic}

	\label{alg}%

\end{algorithm}

\begin{thm}\label{thm-2}
	1) PC\&CG will converge to the global optimum ($\varepsilon=0$) of \textbf{DDU-Single} in a finite number of iterations.
	
	2) The number of iterations before PC\&CG terminates is bounded by ${\binom{n_{\rm \xi}+m_{\rm \xi}}{m_{\rm \xi}}}^{|\mathcal{S}|}$, where $\binom{\cdot}{\cdot}$ represents the combinatorial number and
	\begin{align}
		\binom{n_{\rm \xi}+m_{\rm \xi}}{m_{\rm \xi}}=\frac{(n_{\rm \xi}+m_{\rm \xi})!}{m_{\rm \xi}!\cdot n_{\rm \xi}!},
	\end{align}
	and the iteration complexity of PC\&CG is $\mathcal{O}\left({\binom{n_{\rm \xi}+m_{\rm \xi}}{m_{\rm \xi}}}^{|\mathcal{S}|}\right)$.
\end{thm}

\begin{proof}
	See the proof in Appendix \ref{appen-2}.
\end{proof}

\begin{remark}
	1) Compared to the basic C\&CG \cite{zeng2013solving}, PC\&CG holds similar structure. It iteratively generates new variables $\bm{y}_s^{\hat{\bm{\lambda}}_s}$, $\bm{\xi}_s^{\hat{\bm{\lambda}}_s}$ and $\bm{\vartheta}_s^{\hat{\bm{\lambda}}_s}$ (corresponds to ``columns'') and constraints \eqref{Add-1}--\eqref{Add-3} to strengthen \textbf{MP}. Specifically, it introduces valuable extreme point $\bm{\xi}_s^{\hat{\bm{\lambda}}_s}$ of $\Xi_s(\bm{x})$ parametrically through $\mathcal{OU}_s(\bm{x},\hat{\bm{\lambda}}_s)$ in \eqref{Add-3}, and replicates the corresponding lower-level problem in \eqref{Add-1} and \eqref{Add-2}. Hence, the proposed algorithm is termed ``parametric column-and-constraint generation''.

		2) The $\hat{\mathcal{P}}_s$'s can be initialized through certain  strategies or domain expertise. For example, we can obtain initial $\bm{\lambda}_s$'s by solving \textbf{SP}$_s$'s with a trial $\hat{\bm{x}}\in\mathcal{X}$, which can be given by experts. Moreover, for the sake of simplicity, initial $\hat{\mathcal{P}}_s$'s can be directly chosen as empty sets.
	
	3) In real life, when $\frac{UB-LB}{\left| LB\right|}$ is small, the associated solution is good enough for the practical systems. Hence, rather than taking an extremely long time to derive exact solutions, we can choose a small $\varepsilon$, e.g., 0.1\%, as the optimality tolerance, which terminates the algorithm whenever $\frac{UB-LB}{\left| LB\right|}\leq \varepsilon$ and hence mitigates the computational burden. Note that when $\varepsilon>0$,  $\bm{x}^{\star}$ corresponding to the upper bound upon termination is output as the optimal solution.

	4) For each $s\in\mathcal{S}$, the corresponding \textbf{SP}$_s$ is independent of others. Therefore, it provides us with the potential to perform parallelization in practice, so as to improve the computational performance of PC\&CG.
\end{remark}

\section{Case Study}\label{case}
\subsection{Description}
Numerical tests have been performed on a 33-bus exemplary distribution network \cite{IEEE-33}, as shown in Fig. \ref{33-bus}, to validate the proposed \textbf{DDU-NHEMP} model and the PC\&CG algorithm. The PDN is partitioned into three refueling zones, termed A, B, and C. Parameters of candidate components in hydrogen-electrical microgrids are provided in Tables \ref{tab:electrical} and \ref{tab:hydrogen-subsystem} \cite{hydro11,dis1,osti_1876290,SUN2024111125}. Further, the hydrogen procurement and retail prices are taken as 8\$/kg and 9.304\$/kg, respectively \cite{hydro3,hydro11}.  For every PDN node, the upper ($U_i^{\max}$) and lower ($U_i^{\min}$) bounds of nodal voltage are 10.7kV and 9.3kV, respectively.
\begin{figure}[]
	\centering
	\includegraphics[width=0.47\textwidth]{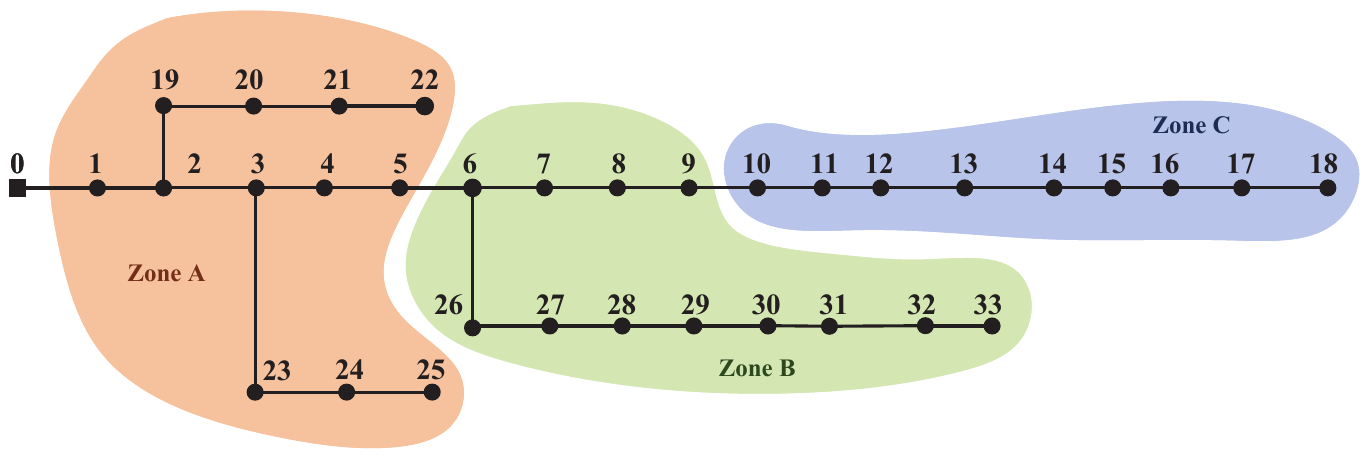}
	\caption{Topology and zone partitioning of 33-bus exemplary network.}
	\label{33-bus}
\end{figure}

\begin{table}[]
  \centering
  \setlength{\tabcolsep}{1mm}{
  \caption{Major Parameters of Components in Electrical Subsystem}
   \label{tab:electrical}
    \begin{tabular}{ccc}
    \toprule
    PV array & WT    & BB \\
    \midrule
    \makecell[c]{$c_{\rm pv}^{\rm i\&m}=153.67$\$/kW/yr\\
$\overline{P}_{\rm pv}=80$kW} & \makecell[c]{$c_{\rm wt}^{\rm i\&m}=225.79$\$/kW/yr\\
$\overline{P}_{\rm wt}=200$kW} & \makecell[c]{$c_{\rm bb}^{\rm i\&m}=51.76$\$/kW/yr\\
$\kappa_{\rm bb}=0.001$\$/kW\\
$\overline{P}_{\rm bb}=90$kW\\
$\overline{E}_{\rm bb}=150$kWh\\
$\eta_{\rm bb}^{\rm ch}=\eta_{\rm bb}^{\rm dis}=90$\%\\
$DoD_{\rm bb}=85$\%} \\
    \bottomrule
    \end{tabular}%
}
\end{table}%

\begin{table}[]
  \centering
  \caption{Major Parameters of Components in Hydrogen Subsystem}
  \label{tab:hydrogen-subsystem}
  \scalebox{0.95}{
    \begin{tabular}{ccc}
    \toprule
    ELZ   & HT    & HD \\
    \midrule
    \makecell[c]{$c_{\rm elz}^{\rm i\&m}=41.05$\$/kW/yr \\
$\overline{P}_{\rm elz}=200$kW \\
$\eta_{\rm elz}=76$\% \\
$LHV_{\rm H_2}=33.33$kW/kg} & \makecell[c]{$c_{\rm ht}^{\rm i\&m}=56.76$\$/kg/yr  \\
$\overline{P}_{\rm ht}=100$kg \\
$\phi_{\rm ht}=2$\%} & \makecell[c]{$c_{\rm hd}^{\rm i\&m}=29974.55$\$/yr \\ 
$SR=108$kg/h} \\
    \bottomrule
    \end{tabular}%
}
\end{table}%

\begin{table}[]
  \centering
  \caption{Induced Coefficients of DDU set}
  \setlength{\tabcolsep}{5mm}{
  	\scalebox{0.75}{
  	\renewcommand{\arraystretch}{1.4}
    \begin{tabular}{ccccccc}
    \toprule
    \multirow{2}[2]{*}{Coefficient} & \multicolumn{6}{c}{Time Period (h)} \\
\cmidrule{2-7}          & 1-4   & 5-8   & 9-12  & 13-16 & 17-20 & 21-24 \\
    \midrule
    $\overline{\gamma}_{A}^t$ (kg/h) & 30    & 40    & 60    & 60    & 40    & 30 \\
    $\underline{\gamma}_{A}^t$ (kg/h) & 25    & 30    & 40    & 40    & 35    & 25 \\
    \midrule
    $\overline{\gamma}_{B}^t$ (kg/h) & 25    & 25    & 45    & 45    & 35    & 30 \\
    $\underline{\gamma}_{B}^t$ (kg/h) & 20    & 20    & 35    & 35    & 25    & 20 \\
    \midrule
    $\overline{\gamma}_{C}^t$ (kg/h) & 20    & 30    & 40    & 40    & 30    & 20 \\
    $\underline{\gamma}_{C}^t$ (kg/h) & 15    & 25    & 30    & 30    & 25    & 15 \\
    \bottomrule
    \end{tabular}%
}
}
  \label{tab:ddu-parameter}%
\end{table}%

Table \ref{tab:ddu-parameter} shows the DDU coefficients $\overline{\gamma}_z^t$ and $\underline{\gamma}_z^t$ of each zone, and we have chosen $\overline{\alpha}^t=\sum_{z\in\mathcal{Z}}\overline{\gamma}_z^{t}/\left|\mathcal{Z} \right| $ and $\underline{\alpha}^t=\sum_{z\in\mathcal{Z}}\underline{\gamma}_z^{t}/\left|\mathcal{Z} \right| $. The DIU factors, i.e., the variations of RES outputs and electric loads, can be described using well-defined probability distributions \cite{graham1990method,zhang2010simulation,du2018scenario}. A hybrid Latin hypercube sampling and $k$-means clustering approach \cite{4808223,1017616} has been leveraged for typical scenario generation, which helps balance the computational efficiency and accuracy.

Our tests have been implemented on a mobile workstation with Intel(R) Core(TM) i9-13900H processor and 16GB RAM. The PC\&CG algorithm was implemented by MATLAB R2023b with YALMIP Toolbox \cite{Lofberg2004} and Gurobi 11.0 solver \cite{gurobi1}. The ``big-M''s for \textbf{MP} and \textbf{SP}$_s$ are respectively set to $10^{6}$ and $10^9$, the termination tolerance $\varepsilon$ is chosen as 0.1\%, and the initial $\hat{\mathcal{P}_s}$'s are chosen as empty sets. The algorithm parallelization is realized by using the Parallel Computing Toolbox in MATLAB.

\subsection{Results of NHEMP}\label{result}
Cases with the same set of three stochastic scenarios are studied below. 
\begin{itemize}
	\item \textit{Case 1:} Deploy hydrogen-electrical microgrids without considering DIE by employing static DIU sets for refueling demand.
	\item \textit{Case 2:} Deploy hydrogen-electrical microgrids considering DIE by employing dynamic DDU sets for refueling demand.
\end{itemize}
By comparing these two cases, we have investigated the necessity and impact of the investment-induced refueling demand in the context of NHEMP.
\begin{figure}[]
	\centering
	\subfloat[Case 1 without DIE: DIU modeling.]{
		\label{scheme-1}
		\includegraphics[width=0.48\textwidth]{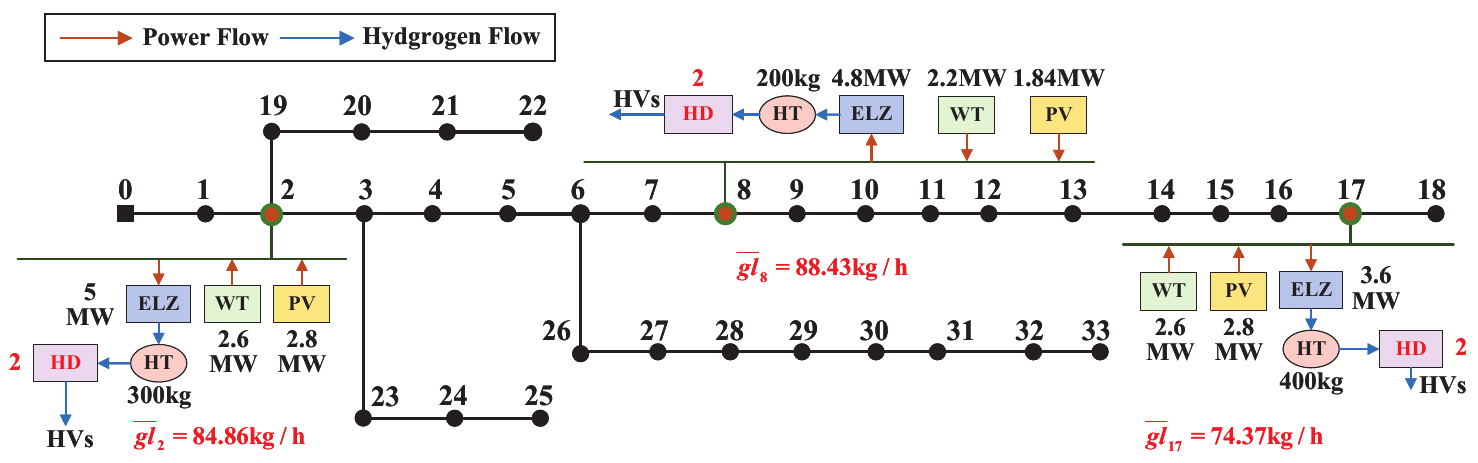}
	}\\
	\subfloat[Case 2 with DIE: DDU modeling.]{
		\label{scheme-2}
		\includegraphics[width=0.48\textwidth]{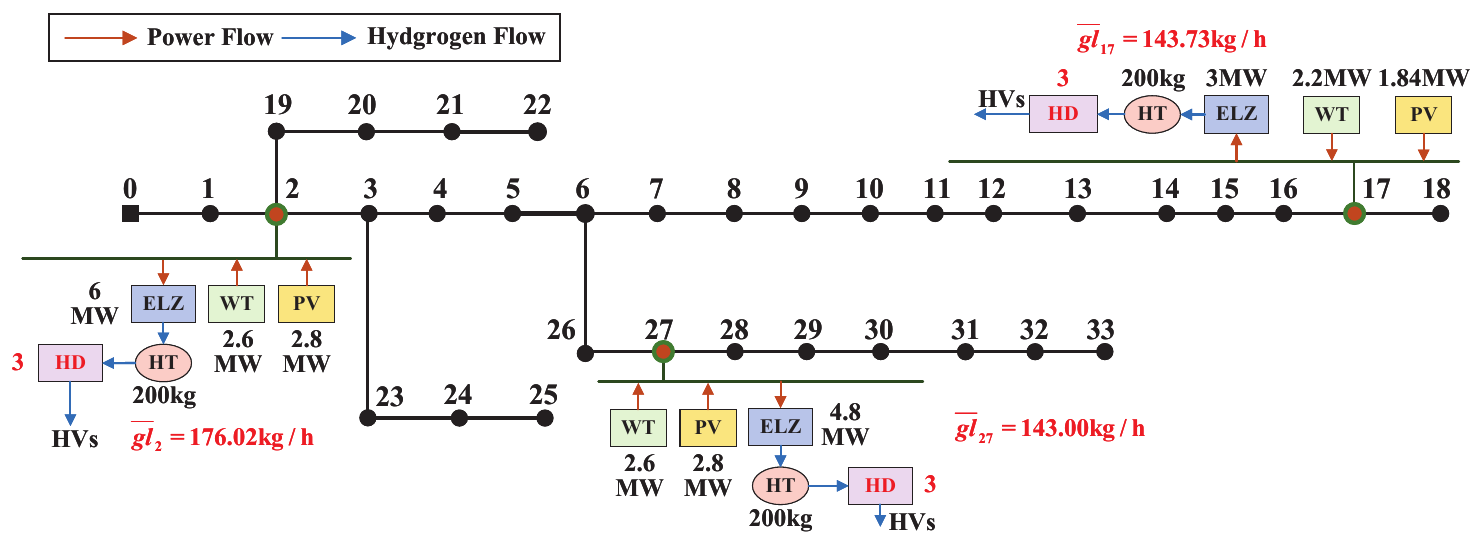}
	}
	\caption{Networked hydrogen-electrical microgrids planning results.}
	\label{scheme}
\end{figure}
 
Fig. \ref{scheme} depicts the planning results of these two cases, and the average values of nodal met refueling demand ($\overline{gl}_{i}$) are also exhibited, where $\overline{gl}_{i}$ can be calculated by
\begin{align}
	\overline{gl}_i=\frac{\sum_{s\in\mathcal{S}}\sum_{t\in\mathcal{T}}gl_i^{s,t}}{|\mathcal{S}||\mathcal{T}|}, \quad \forall i\in\Lambda.\label{}
\end{align}
It can be seen that the number of HDs are, respectively, 6 and 9 for Case 1 and Case 2, and the corresponding average met refueling demand over the entire system ($\overline{gl}$), which can be defined as
\begin{align}
	\overline{gl}=\sum_{i\in\Lambda}\overline{gl}_i,\label{}
\end{align}
are 247.66kg/h and 453.75kg/h, respectively. Also, a hydrogen-electrical microgrid at Node 8 (for Case 1) is switched to Node 27 (for Case 2) with a larger resource capacity. Compared to Case 1, the $\overline{gl}$ of Case 2 has an increase of 206.09kg/h, clearly reflecting the demand-inducing effect. It can be explained that the deployment of HDs induces the growth of refueling demand, which in turn fosters the additional investments of hydrogen-electrical microgrids. Eventually, an equilibrium state between the supply and demand sides is reached. 

\begin{table}[]
	\centering
	\caption{Economic-Benefits Assessment of NHEMP}
	\setlength{\tabcolsep}{2mm}{
		\begin{tabular}{cccccc}
			\toprule
			\multirow{2}[0]{*}{Index} & \multirow{2}[0]{*}{\makecell[c]{$\Phi$\\
					(million\$)}} & \multirow{2}[0]{*}{\makecell[c]{$\Phi^{\rm capex}$\\
					(million\$)}} & \multirow{2}[0]{*}{\makecell[c]{$\Phi^{\rm o\&m}$\\
					(million\$)}}  & \multirow{2}[0]{*}{\makecell[c]{$V_{\rm DIE}$\\
					(million\$)}} & \multirow{2}[0]{*}{$\overline{V}_{\rm DIE}$} \\
			&       &       &       &          &  \\
			\hline
			Case 1 & -10.63 & 3.76  & -14.40      & \textbackslash{}     & \textbackslash{} \\
			Case 1\# & -12.41 & 3.76  & -16.17      & \textbackslash{}     & \textbackslash{} \\
			Case 2 & -13.01 & 3.85  & -16.86      & 0.6     & 4.83\% \\
			\bottomrule
		\end{tabular}%
	}
	\label{economic}%
\end{table}%

Next, we have investigated the economic benefits of microgrids deployment. As shown in Table \ref{economic}, three annualized indices, i.e., ASE ($\Phi$), CapEx ($\Phi^{\rm capex}$), and summation of OPEX and maintenance cost ($\Phi^{\rm o\&m}$), are selected. In Case 2, $\Phi$ is -13.01 million\$, supporting the economic feasibility for investing in hydrogen-electrical microgrids, since a negative $\Phi$ represents a positive profit over the planning horizon. Compared to Case 1, Case 2 has a 17.08\% decline in $\Phi^{\rm o\&m}$, while only with a 2.39\% increase in $\Phi^{\rm capex}$. It shows that the DDU offers \textit{flexibility} to NHEMP modeling. Specifically, in trilevel \textbf{DDU-NHEMP}, the middle-level problem seeks for the refueling demand as small as possible to reduce microgrid operator's revenues. With an embedded DDU set, the ``worst-case'' refueling demand will increase with the refueling capacity, thereby improving the profits. 

We provide Case 1\# to represent the expected economic benefits of Case 1 in an environment with DIE, which is calculated by fixing the investment strategies as the corresponding solutions to Case 1 and re-solving the DDU model. Comparing Case 1\# with Case 1, the induced refueling demand is expected to result in a 1.78 million\$ decrease in $\Phi$, suggesting that the economic benefit is underestimated if our planning results are derived by assuming a DIU set. Indeed, we note that such an underestimation may cause us to give up an economic-feasible project in practice.

To formalize our understanding, we introduce two concepts: the \textit{Value of Demand-Inducing Effect} ($V_{\rm DIE}$) and the \textit{Relative Value of Demand-Inducing Effect} ($\overline{V}_{\rm DIE}$), to quantify the benefit of incorporating DIE into the decision-making model. Mathematically, they are defined as:
\begin{eqnarray}
	&&V_{\rm DIE}=\Phi_{1\#}-\Phi_{2}, \label{VDDU-1}\\
	&&\overline{V}_{\rm DIE}= \frac{{\Phi}_{1\#}-\Phi_{2}}{\left|{\Phi}_{1\#}\right|} \times100\%, \label{VDDU-2}
\end{eqnarray}
where $\Phi_{1\#}$ and $\Phi_{2}$ represent the ASE of Case 1\# and Case 2, respectively. Based on their definitions and noting that the result of Case 1\# is just a feasible plan for Case 2, the next result readily follows.  
\begin{prop}
$\Phi_{1\#}\geq \Phi_{2}$, i.e., $V_{\rm DIE}\geq 0$ and $\overline{V}_{\rm DIE}\geq 0$. 
\end{prop}

As in Table \ref{economic}, a 0.6 million\$ $V_{\rm DIE}$ and a 4.83\% $\overline{V}_{\rm DIE}$ are reported. Actually, given that the $\Phi^{\rm capex}$ increase from Case 1\# to Case 2 is rather marginal, which is about 0.09 million\$, we can safely attribute most of these nontrivial benefits to the sound investment plan derived from the DIE-aware NHEMP model. Certainly, as shown in the next subsection, the impact of DIE requires a more systematic study.

Finally, we have studied voltage levels. The maximum value (MaxU), minimum value (MinU), average value (AveU), and variance value (VarU) of voltage are defined as follows:
\begin{align}
	&{\rm MaxU}=\max\left\{U_i^{s,t}\right\}_{s\in\mathcal{S},i\in\mathcal{N^+},t\in\mathcal{T}},\\
	&{\rm MinU}=\min\left\{U_i^{s,t}\right\}_{s\in\mathcal{S},i\in\mathcal{N^+},t\in\mathcal{T}},\\
	&{\rm AveU}=\frac{\sum_{s\in\mathcal{S}}\sum_{t\in\mathcal{T}}\sum_{i\in\mathcal{N^+}}U_i^{s,t}}{\left| \mathcal{S}\right|\left| \mathcal{T}\right|\left| \mathcal{N^+}\right| },\\
	&{\rm VarU}=\frac{\sum_{s\in\mathcal{S}}\sum_{t\in\mathcal{T}}\sum_{i\in\mathcal{N^+}}(U_i^{s,t}-{\rm AveU})^2}{\left| \mathcal{S}\right|\left| \mathcal{T}\right|\left| \mathcal{N^+}\right|-1 }.
\end{align}
As shown in Table \ref{voltage}, the voltage magnitudes of all nodes fall into well the secure range [9.3kV,~10.7kV], indicating the system is able to work well within those voltage bounds, regardless of whether the HVs' refueling demand is captured by DIU or DDU sets.  The relative small variance also suggests that system operations are rather stable under those uncertainty sets.  Certainly, we mention that such observations are  rather system or parameter specific, as different voltage bounds can be adopted that may become critical factors in determining the system planning and operation solutions.

\begin{table}[]
	\centering
	\caption{Voltage Level Results of NHEMP}
		\setlength{\tabcolsep}{2.6mm}{
	\begin{tabular}{ccccc}
		\toprule
		Index & MaxU (kV) & MinU (kV) & AveU (kV) & VarU (kV$^2$)\\
		\midrule
		Case 1 & 10.5716 & 9.4135 & 10.2297 & 0.0423\\
		Case 1\# & 10.4833 & 9.4811 & 10.2184 & 0.0421\\
		Case 2 & 10.5745 & 9.6034 & 10.2607 & 0.0345\\
		\bottomrule
	\end{tabular}%
	\label{voltage}%
}
\end{table}%

\subsection{Sensitivity of DIE}
We have further conducted sensitivity analysis of DIE based on the  parameters of DDU for Case 2. For simplicity, constraints \eqref{ddu-set-3} and \eqref{ddu-set-4} are omitted. In the DDU set, $\overline{\gamma}_z^t$ and $\underline{\gamma}_z^t$ quantify the \textit{unit induced capacity}. Let $\overline{\gamma}_z^t=\chi\cdot\overline{\gamma}_z^{t,0}$ and $\underline{\gamma}_z^t=\chi\cdot\underline{\gamma}_z^{t,0}$, where $\chi\in\left[ 0,1\right] $, and $\overline{\gamma}_z^{t,0}$ and $\underline{\gamma}_z^{t,0}$ are the default values as in Table \ref{tab:ddu-parameter}. Fig. \ref{Economic-DDU} depicts ASE and $\overline{gl}$ with various $\chi$. For $\overline{gl}$, it gradually increases when $\chi$ ranges from 0 to 0.5, reflecting the increase of the system refueling capacity. In this case, the marginal effect on $\overline{gl}$ is the same (increase of 40.83kg/h per 0.1 on $\chi$). When $\chi\geq 0.7$, $\overline{gl}$ is nearly stable around 455.51kg/h, which mainly attributes to the system's inherent restrictions on hydrogen production, storage, and refueling. On the other hand, the ASE demonstrates a  negative correlation with $\chi$ decreasing from -10.21 to -13.01 million\$. The marginal value of ASE generally decreases with $\chi$ (except from 0 to 0.1). Corresponding to the growing trend of $\overline{gl}$, ASE experiences a rapid decline of 2.21 million\$ as $\chi$ ranges from 0 to 0.5, while for $\chi$ ranging from 0.5 to 1, the decrease in ASE is only 0.59 million\$. \textit{It suggests that as $\overline{\gamma}_z^t$ and $\underline{\gamma}_z^t$ increase, the microgrids deployment indeed achieves greater gains, but eventually some physical restrictions will impede further profitability.} Moreover, Fig. \ref{Value-DDU} shows the relationship between $\overline{V}_{\rm DIE}$ and $\chi$, and the $\overline{gl}$ curve is also given as a reference. The variation of $\overline{V}_{\rm DIE}$ shows an initial increase (0 to 13.23\%) with $\chi$, followed by a subsequent decrease from 13.23\% to 9.70\% after $\chi=0.6$. Note that the decline emerges after $\overline{gl}$ becomes stable. We can conclude that $\overline{\gamma}_z^t/\underline{\gamma}_z^t$ positively contributes to $\overline{V}_{\rm DIE}$, and this impact will be gradually reduced when other system constraints become the bottleneck.

\begin{figure}[]
	\centering
	\subfloat[Analysis of ASE and $\overline{gl}$.]{
		\label{Economic-DDU}
		\includegraphics[width=0.47\textwidth]{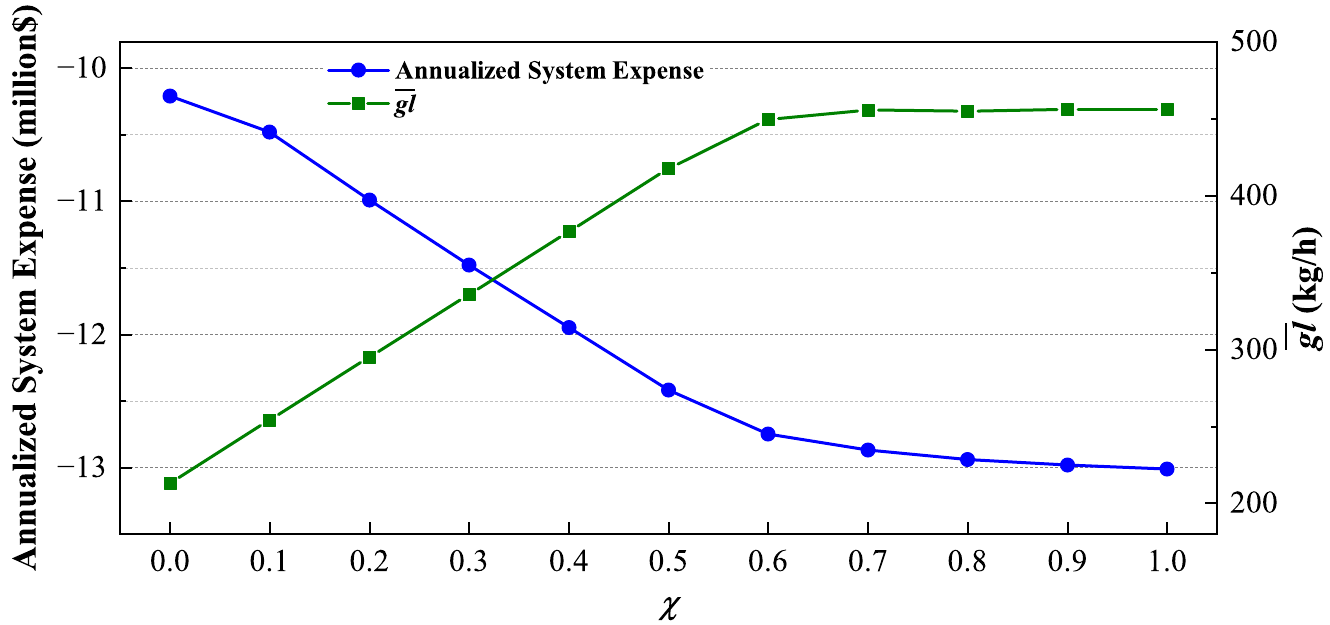}
	}\\
	\subfloat[Analysis of $\overline{V}_{\rm{DIE}}$.]{
		\label{Value-DDU}
		\includegraphics[width=0.47\textwidth]{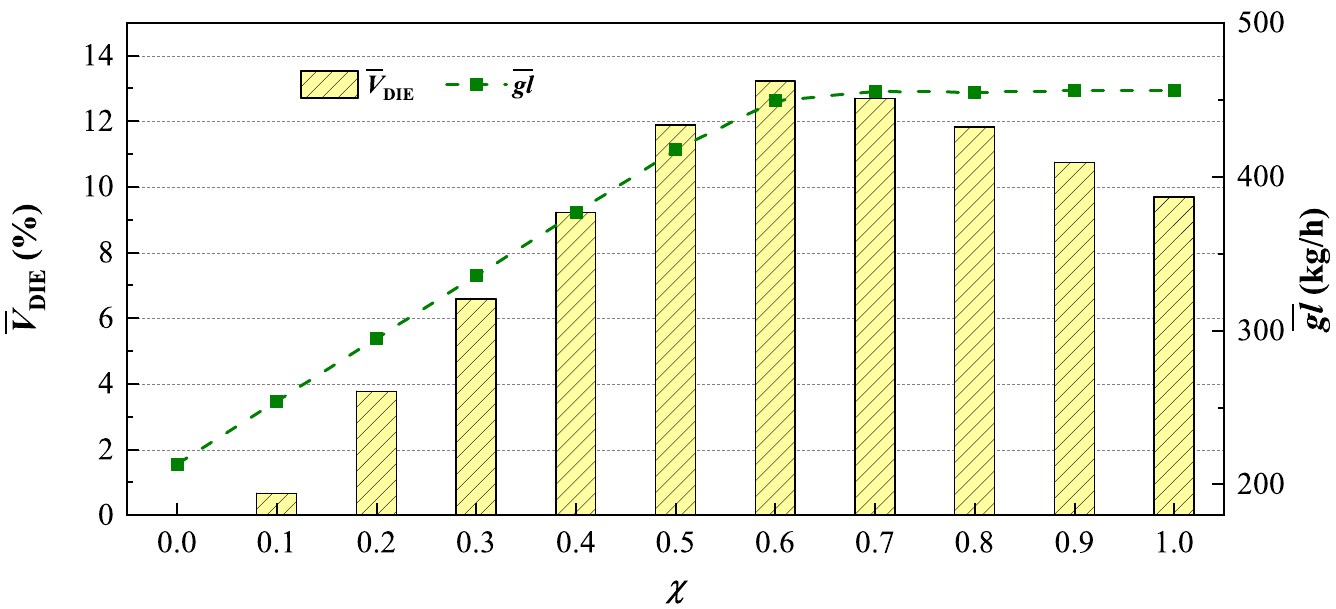}
	}
	\caption{Sensitivity of demand-inducing effect.}
	\label{SA}
\end{figure}

\begin{remark}
	Results in Fig. \ref{SA} indicate that the value of DIE are significantly influenced by the induced coefficients. Therefore, it is critical to make accurate predictions regarding $\overline{\gamma}_z^t/\underline{\gamma}_z^t$ and $\overline{\alpha}^t/\underline{\alpha}^t$, which actually are correlated with the subjective behaviors of drivers, and conduct exact numerical computation to leverage the modeling capacity of DDU, thereby unleashing the economic advantages of DIE.
\end{remark}

\subsection{Computational Performance of PC\&CG}
The computational tests of our customized PC\&CG to solve \textbf{DDU-NHEMP} have been performed and analyzed. 

We have compared PC\&CG with another method for the trilevel $\min-\max-\min$ problem with DDU in the literature, i.e., Benders C\&CG (BC\&CG) \cite{zeng2022two}, which is basically equivalent to the modified Benders dual decomposition presented in Ref. \cite{9754328}. Analogous to our PC\&CG algorithm, original BC\&CG is also extended to the situation with multiple scenarios. For each scenario $s$, with $\hat{\bm{\lambda}}_s$ from \textbf{SP}$_s$, BC\&CG creates new variables $\bm{\xi}_s^{\hat{\bm{\lambda}}_s}$ and $\bm{\vartheta}_s^{\hat{\bm{\lambda}}_s}$, and adds constraints \eqref{bccg-1} and \eqref{bccg-2} as the optimality cut to strengthen \textbf{MP}.
\begin{align}
	&\eta_s\geq (\bm{f}_s-\bm{G}_s\bm{x}-\bm{E}_s\bm{\xi}_s^{\hat{\bm{\lambda}}_s})^{\intercal}\hat{\bm{\lambda}}_s,\quad s\in\mathcal{S}\label{bccg-1}\\
	&(\bm{\xi}_s^{\hat{\bm{\lambda}}_s},\bm{\vartheta}_s^{\hat{\bm{\lambda}}_s})\in\mathcal{OU}_s(\bm{x},\hat{\bm{\lambda}}_s),\quad s\in\mathcal{S}\label{bccg-2}
\end{align}
Fig. \ref{converge} exhibits the convergence behaviors of PC\&CG and BC\&CG under the parameter setting in Case 2. Both of them ultimately converge to the same objective value. However, PC\&CG is almost 180 times faster than BC\&CG, demonstrating a clear dominance over the latter one in computational efficiency. Specifically, PC\&CG could obtain initial upper and lower bounds with exceptional quality (only 2.72\% gap). Then, PC\&CG will quickly diminish the gap, and achieve the convergence in only 4 iterations. On the contrary, BC\&CG starts with UB ($-1.22\times10^{7}$) and LB ($-7.68\times10^{11}$) with a very huge gap (nearly 100\%). After a few iterations from the beginning, UB and LB improve very slowly, leading to a large number of iterations before termination, thereby much longer solution time. 

In fact, for scenario $s$, both PC\&CG and BC\&CG iteratively introduce non-trivial $\bm{\xi}_s^{\hat{\bm{\lambda}}_s}$ through its optimality conditions in a parametric manner, as in \eqref{Add-3} and \eqref{bccg-2}. The difference is that PC\&CG generates a replication of the whole lower-level problem associated with $\bm{\xi}_s^{\hat{\bm{\lambda}}_s}$ to \textbf{MP}, in the form of \eqref{Add-1} and \eqref{Add-2}, building a piecewise-linear under-approximation to the \textit{value function} $\mathcal{Q}_s(\bm{x})\triangleq\max_{\bm{\xi}_s\in\Xi_s(\bm{x})}\min_{\bm{y}_s\in\mathcal{Y}_s(\bm{x},\bm{\xi}_s)}\bm{d}_s^{\intercal}\bm{y}_s$; while BC\&CG only adds a dual hyperplane \eqref{bccg-1} associated with $\bm{\xi}_s^{\hat{\bm{\lambda}}_s}$ to under-approximate $\mathcal{Q}_s(\bm{x})$. Compared to BC\&CG, in one iteration, PC\&CG provides a stronger approximation to $\mathcal{Q}_s(\bm{x})$, thereby resulting in a better convergence behavior. Given the drastic differences between these two methods, we do not recommend BC\&CG to compute practical instances. 

\begin{remark}
	We mention that PC\&CG (resp. BC\&CG) generalizes C\&CG (resp. the Benders-dual method) \cite{zeng2013solving}, which are popular algorithms for the solution of trilevel problems with DIU. Our result further confirms and extends the dominance of C\&CG over the Benders-dual method, e.g., in Refs. \cite{zeng2013solving,bruni2018computational,lima2022risk}, in the DDU environment.
\end{remark}

\begin{figure}[]
		\centering
		\includegraphics[width=0.47\textwidth]{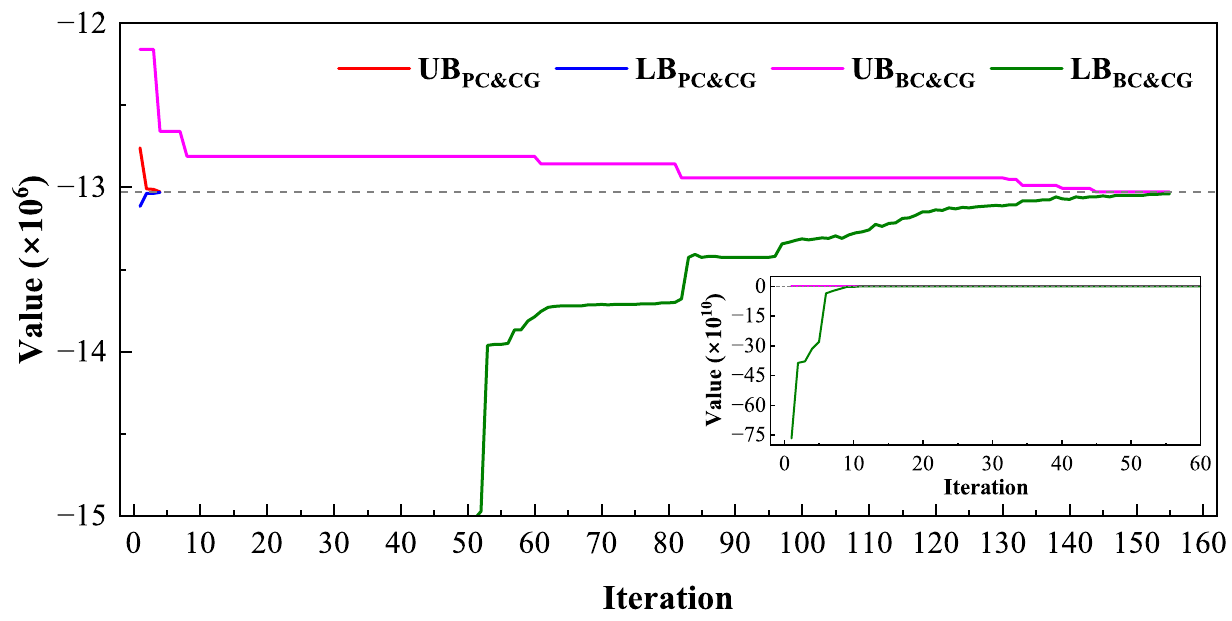}
		
		\caption{Convergence behaviors of PC\&CG and BC\&CG. ($|\mathcal{S}|=3$. Solution times of PC\&CG and BC\&CG are \textbf{41.62s} and \textbf{7410.30s}, respectively, and the corresponding numbers of iterations are \textbf{4} and \textbf{155}.)}
		\label{converge}
\end{figure}

\begin{table}[]
	\centering
	\caption{Computational Results of PC\&CG}
	\begin{tabular}{cccccc}
		\toprule
		\multirow{2}[2]{*}{$|\mathcal{S}|$} & \multicolumn{5}{c}{PC\&CG} \\
		\cmidrule{2-6}          & UB    & LB    & Gap   & Time (s) & Iter \\
		\midrule
		1     & -12663447.38  & -12675944.60  & 0.10\% & 5.20  & 2 \\
		2     & -12720282.10  & -12731305.32  & 0.09\% & 12.46  & 3 \\
		3     & -13027034.59  & -13030024.49  & 0.02\% & 41.62  & 4 \\
		4     & -13583619.00  & -13583619.51  & $<$0.01\% & 55.68  & 4 \\
		5     & -13600526.07  & -13611305.00  & 0.08\% & 15.09  & 2 \\
		6     & -13232295.70  & -13245070.73  & 0.10\% & 38.10  & 3 \\
		7     & -12944700.11  & -12954826.34  & 0.08\% & 36.88  & 3 \\
		8     & -12731237.61  & -12731237.89  & $<$0.01\% & 66.42  & 4 \\
		9     & -12572078.93  & -12575423.70  & 0.03\% & 113.70  & 4 \\
		10    & -12348591.05  & -12350884.49  & 0.02\% & 165.18  & 3 \\
		\bottomrule
	\end{tabular}%
	\label{tab3}%
\end{table}%

Moreover, we have investigated the scalability of PC\&CG. The computational results using PC\&CG upon \textbf{DDU-NHEMP} with various numbers of scenarios ($|\mathcal{S}|$) are presented in Table \ref{tab3}. For each instance, we show the UB and LB, relative gap (Gap), number of iterations (Iter), and computational time in seconds (Time) of PC\&CG when it terminates. For all the instances with $|\mathcal{S}|$ ranging from 1 to 10, PC\&CG can solve them in 4 iterations, and lead to high-quality solutions within a 0.10\% gap. Even for the complex instance with 10 scenarios, it can be solved within 3 minutes. We notice that the computational time is not always increased with $|\mathcal{S}|$ monotonically (e.g., instances with $|\mathcal{S}|=5,~6$ and $7$). It also depends on the operating parameters for each scenario, which may affect the number of required iterations. Intuitively, for instances with the same number of iterations, the solution time increases with $|\mathcal{S}|$. But cases with 6 and 7 scenarios are exceptional, which spend similar amount of time for convergence. In summary, PC\&CG offers a powerful and scalable computational tool to solve complex \textbf{DDU-NHEMP} problems.

\section{Discussion}\label{discussion}
The previous section has provided case studies on NHEMP with DIE. Together with the numerical results, here, we summarize and include some discussions on our key findings.
\begin{enumerate}
	\item\textit{Demand-Inducing Effect:} DIE positively contributes to the economic benefits of NHEMP.  The case studies have indicated that incorporating the consideration of DIE into the decision-making model can result in a maximum of 13.23\% gain in expected benefits. However, this impact will be gradually reduced when other system restrictions become the bottleneck. 
	\item\textit{Decision-Dependent Uncertainty:} DDU offers a flexible modeling capacity to generate formulations. For the NHEMP problem, DDU can capture the investment induced refueling demand, and thus accurately, if solved to optimality, reflects the impact of DIE and helps us derive a sound investment plan. 
	\item\textit{Parametric Column-and-Constraint Generation Algorithm:} PC\&CG provides a powerful and scalable computational tool. It can exactly and efficiently solve the complex trilevel stochastic-robust NHEMP problem with DDU. The results of our numerical experiments demonstrate the clear domination of PC\&CG over another dual cutting-plane method, i.e., BC\&CG, with a remarkable reduction in solution time by almost 180 times. Moreover, PC\&CG has the potential to be employed solving similar types of practical problems.
\end{enumerate}

\section{Conclusion and Future Direction}\label{conclusion}
This paper has proposed a stochastic-robust approach to support NHEMP with DIE. A novel $\min-\max-\min$ formulation with DDU sets has been presented, which reflects the mutual interactions between investment decisions and HV refueling demand. To address the computational intractability, an adaptive decomposition algorithm based on PC\&CG has been customized and developed. Case studies on an IEEE exemplary system have corroborated the effectiveness of our DIE-aware NHEMP model and the advancement of our PC\&CG algorithm.

There are several interesting directions worth investigating for future research. One potential work is to introduce DIE into the realm of demand response, thereby applying the proposed DDU-based decision framework to the optimal operation of smart grids. Particularly, problems with market equilibria are of our interest. Furthermore, it is worth studying how  PC\&CG can be leveraged to handle the trilevel stochastic-robust problems whose middle- and lower-levels contain integer-valued variables and nonlinearity.

\appendices
\section{Proof of Lemma \ref{lemma-1}} \label{appen-1}
\begin{proof}[Proof of Lemma \ref{lemma-1}]
On the one hand, given a  $\bm{\lambda}_s^{\prime}\in\mathcal{P}_s$, according to the definitions of $\mathcal{OU}_s$ (in Section \ref{def-OU}) and the projection relationship between $\mathcal{OU}_s^{\bm{\xi}_s}$ and $\mathcal{OU}_s$ (defined in Lemma \ref{lemma-1}), $\mathcal{OU}_s^{\bm{\xi}_s}(\bm{x},\bm{\lambda}_s^{\prime})\subseteq\Xi_s(\bm{x})$ holds. Hence, considering all $\bm{\lambda}_s$'s in $\mathcal{P}_s$, we have $\textstyle\bigcup_{\bm{\lambda}_s\in\mathcal{P}_s}\mathcal{OU}_s^{\bm{\xi}_s}(\bm{x},\bm{\lambda}_s)=\Xi_s^*(\bm{x})\subseteq\Xi_s(\bm{x})$, and $LHS\geq RHS$.

On the other hand, following Proposition \ref{prop-1} in Section \ref{def-OU}, there exists an optimal $\hat{\bm{\xi}}_s$ of the LHS also an optimal solution of $\bm{\Psi}_s(\bm{x},\hat{\bm{\lambda}}_s)$, with a certain $\hat{\bm{\lambda}}_s\in\mathcal{P}_s$. Recalling the definitions of $\mathcal{OU}_s$ and $\mathcal{OU}_s^{\bm{\xi}_s}$, we have $\hat{\bm{\xi}}_s\in\mathcal{OU}_s^{\bm{\xi}_s}(\bm{x},\hat{\bm{\lambda}}_s)$, and thus $\hat{\bm{\xi}}_s\in\Xi_s^*(\bm{x})$. Hence, $LHS\leq RHS$ holds. 

Together with $LHS\geq RHS$ and $LHS\leq RHS$, $LHS=RHS$ readily follows.
\end{proof}

\section{Proof of Theorem \ref{thm-2}}\label{appen-2}
It is worth highlighting that the proof follows the core theory of linear programming \cite{karloff2008linear}. Firstly, we introduce the standard LP form of $\bm{\Psi}_s(\bm{x},\bm{\lambda}_s)$, for  given $\bm{\lambda}_s\in\Pi_s$ and $s\in\mathcal{S}$, which can be derived by adding $m_{\rm \xi}$ non-negative slack variables to original $\bm{\Psi}_s(\bm{x},\bm{\lambda}_s)$:
\begin{eqnarray}
	&&\widetilde{\bm{\Psi}}_s(\bm{x},\bm{\lambda}_s):\min_{\widetilde{\bm{\xi}}_s\in\widetilde{\Xi}_s(\bm{x})}\left(\widetilde{\bm{E}}_s\widetilde{\bm{\xi}}_s\right)^{\intercal}\bm{\lambda}_s=\left(\mathfrak{e}_s^{\bm{\lambda}_s}\right)^{\intercal}\widetilde{\bm{\xi}}_s,\\
	&&\widetilde{\Xi}_s(\bm{x})=\left\{\widetilde{\bm{\xi}}_s\in\mathbb{R}_+^{n_{\rm \xi}+m_{\rm \xi}}\;\Big|\;\widetilde{\bm{H}}_s\widetilde{\bm{\xi}}_s=\bm{h}_s-\bm{F}_s\bm{x}\right\},\label{}
\end{eqnarray}
where $\mathfrak{e}_s^{\bm{\lambda}_s}=(\widetilde{\bm{E}}_s)^{\intercal}\bm{\lambda}_s$. With a $\widetilde{\bm{\xi}}_s$ from $\widetilde{\bm{\Psi}}_s(\bm{x},\bm{\lambda}_s)$, the corresponding $\bm{\xi}_s$ can be recovered by retaining the 1-st to $n_{\rm \xi}$-th elements of $\widetilde{\bm{\xi}}_s$. Obviously, $\widetilde{\bm{H}}_s$ is an $m_{\rm \xi}\times(n_{\rm \xi}+m_{\rm \xi})$ matrix, the $m_{\rm \xi}$ rows of which are linearly independent. Hence,  in $\widetilde{\bm{H}}_s$, there must exist $m_{\rm \xi}$ linearly independent column vectors that constitute a \textit{basis}. Note that although $\widetilde{\Xi}_s(\bm{x})$ is dependent on $\bm{x}$, the bases of $\widetilde{\bm{H}}_s$ are independent of $\bm{x}$ and $\bm{\lambda}_s$. Thus, an important lemma follows:

\begin{lem}\label{lemma-2}
	For given $\bm{x}'\in\mathcal{X}$ and $s\in\mathcal{S}$,  consider $\widetilde{\bm{\Psi}}_s(\bm{x}',\bm{\lambda}_s)$ with a fixed $\bm{\lambda}_s\in\Pi_s$. Let $\mathfrak{B}_{s}^{\prime}$ denote a basis of $\widetilde{\bm{H}}_s$, and suppose that its corresponding basic solution (BS) is both a basic feasible solution and an optimal solution to $\widetilde{\bm{\Psi}}_s(\bm{x}',\bm{\lambda}_s)$, i.e., $\mathfrak{B}_{s}^{\prime}$ is an optimal basis and the BS is a basic optimal solution (BOS). If the BS of $\mathfrak{B}_{s}^{\prime}$ in relation to $\widetilde{\bm{\Psi}}_s(\bm{x}'',\bm{\lambda}_s)$~($\bm{x}''\neq\bm{x}'$ and $\bm{x}''\in\mathcal{X}$) is also feasible, it is optimal to $\widetilde{\bm{\Psi}}_s(\bm{x}'',\bm{\lambda}_s)$. Further, if $\mathfrak{B}_{s}^{\prime}$ yields the unique BOS to $\widetilde{\bm{\Psi}}_s(\bm{x}',\bm{\lambda}_s)$, it also yields the unique one to $\widetilde{\bm{\Psi}}_s(\bm{x}'',\bm{\lambda}_s)$.
\end{lem}
\begin{proof}[Proof of Lemma \ref{lemma-2}]
For $s\in\mathcal{S}$, with basis $\mathfrak{B}_{s}^{\prime}$ of $\widetilde{\bm{H}}_s$, we can reorder $\widetilde{\bm{H}}_s=[\mathfrak{B}_{s}^{\prime},\mathfrak{N}_{s}^{\prime}]$ and $\mathfrak{e}_{s}=\begin{bmatrix}
	\mathfrak{e}^{\bm{\lambda}_s}_{s,\mathfrak{B}_{s}^{\prime}}\\ 
	\mathfrak{e}^{\bm{\lambda}_s}_{s,\mathfrak{N}_{s}^{\prime}}
\end{bmatrix}$. 
For $\widetilde{\bm{\Psi}}_s(\bm{x}',\bm{\lambda}_s)$ with fixed $\bm{\lambda}_s$, given the optimality of $\mathfrak{B}_{s}^{\prime}$, the reduced cost  $\bm{r}_{s,\mathfrak{N}_{s}^{\prime}}^{\intercal}=(\mathfrak{e}^{\bm{\lambda}_s}_{s,\mathfrak{N}_{s}^{\prime}})^{\intercal}-(\mathfrak{e}^{\bm{\lambda}_s}_{s,\mathfrak{B}_{s}^{\prime}})^{\intercal}(\mathfrak{B}_{s}^{\prime})^{-1}\mathfrak{N}_{s}^{\prime}$ satisfies $\bm{r}_{s,\mathfrak{N}_{s}^{\prime}}\geq\bm{0}$. Note that $\bm{r}_{s,\mathfrak{N}_{s}^{\prime}}$ is independent of $\bm{x}'$, i.e., $\forall\bm{x}''\in\mathcal{X}$, this inequality always holds. Hence, the BS of $\mathfrak{B}_{s}^{\prime}$ with respect to $\widetilde{\bm{\Psi}}_s(\bm{x}'',\bm{\lambda}_s)$ is also a BOS if, and only if, it is a basic feasible solution. Similarly, the second conclusion readily follows with $\bm{r}_{s,\mathfrak{N}_{s}^{\prime}}>\bm{0}$ strictly.
\end{proof}

Then, the following claim goes to support Theorem \ref{thm-2}. 
\begin{cla}\label{cla-1}
	Let $\hat{\bm{x}}^{\tau}$ and  $(\hat{\bm{\lambda}}^{\tau},\cdots,\hat{\bm{\lambda}}_{|\mathcal{S}|}^{\tau})$ be, respectively, the optimal upper-level solution and extreme points of $\Pi_s$'s obtained in iteration $\tau$ of PC\&CG, and $\mathfrak{B}_{s}^{\tau}~(\forall s\in\mathcal{S})$ denote the optimal basis of $\widetilde{\bm{\Psi}}_s(\hat{\bm{x}}^{\tau},\hat{\bm{\lambda}}^{\tau}_s)$ that corresponds to a BOS. For iteration $\tau_1$ and $\tau_2$ ($\tau_1<\tau_2$), if 
	$(\mathfrak{B}_{1}^{\tau_1},\cdots,\mathfrak{B}^{\tau_1}_{|\mathcal{S}|})=(\mathfrak{B}_{1}^{\tau_2},\cdots,\mathfrak{B}^{\tau_2}_{|\mathcal{S}|})$, $UB=LB$ must hold in iteration $\tau_2$.
\end{cla}
\begin{proof}[Proof of Claim \ref{cla-1}]
	For a practical \textbf{DDU-NHEMP} problem, the optimal solution of $\widetilde{\bm{\Psi}}_s(\bm{x},\bm{\lambda}_s)$ with certain $\bm{x}$ and $\bm{\lambda}_s$ is generally unique. In iteration $\tau_2$, for each $s\in\mathcal{S}$, we use
	 $\widetilde{\bm{\xi}}_s^{\tau_2}=\begin{bmatrix}
	 	(\mathfrak{B}_{s}^{\tau_2})^{-1}(\bm{h}_s-\bm{F}\hat{\bm{x}}^{\tau_2})\\ 
	 	\bm{0}
	 \end{bmatrix}\geq\bm{0}$ 
	denoting the BOS to $\widetilde{\bm{\Psi}}_s(\hat{\bm{x}}^{\tau_2},\hat{\bm{\lambda}}^{\tau_2}_s)$ with $\mathfrak{B}_{s}^{\tau_2}$, and $\hat{\bm{\xi}}_s^{\tau_2}$ being $\widetilde{\bm{\xi}}_s^{\tau_2}$'s incarnation in $\Xi_s(\hat{\bm{x}}^{\tau_2})$. Note that $\widetilde{\bm{\xi}}_s^{\tau_2}$ is independent of $\hat{\bm{\lambda}}_s^{\tau_2}$. Thus,
	\setlength{\arraycolsep}{0.5em}
	\begin{eqnarray}
		\max_{\bm{\xi}_s\in\Xi_s(\hat{\bm{x}}_2)}\min_{\bm{y}_s\in\mathcal{Y}_s(\hat{\bm{x}}_2,\bm{\xi}_s)}\bm{d}_s^{\intercal}\bm{y}_s=\min_{\bm{y}_s\in\mathcal{Y}_s(\hat{\bm{x}}_2,\hat{\bm{\xi}}^{\tau_2}_s)}\bm{d}_s^{\intercal}\bm{y}_s.\label{www}
	\end{eqnarray}
	
	Shift perspective to iteration $\tau_1$: Recalling the uniqueness condition, for each $s$,  $\mathfrak{B}_{s}^{\tau_1}$ yields the unique BOS to $\widetilde{\bm{\Psi}}_s(\hat{\bm{x}}^{\tau_1},\hat{\bm{\lambda}}^{\tau_1}_s)$. Since $(\mathfrak{B}_{1}^{\tau_1},\cdots,\mathfrak{B}^{\tau_1}_{|\mathcal{S}|})=(\mathfrak{B}_{1}^{\tau_2},\cdots,\mathfrak{B}^{\tau_2}_{|\mathcal{S}|})$, according to Lemma \ref{lemma-2}, $\mathfrak{B}_{s}^{\tau_2}$ also yields the unique BOS to $\widetilde{\bm{\Psi}}_s(\hat{\bm{x}}^{\tau_2},\hat{\bm{\lambda}}^{\tau_1}_s)$, i.e., $\begin{bmatrix}
		(\mathfrak{B}_{s}^{\tau_2})^{-1}(\bm{h}_s-\bm{F}\hat{\bm{x}}^{\tau_2})\\ 
		\bm{0}
	\end{bmatrix}=\widetilde{\bm{\xi}}_s^{\tau_2}$. Hence, $\mathcal{OU}_s^{\bm{\xi}_s}(\hat{\bm{x}}^{\tau_2},\hat{\bm{\lambda}}^{\tau_1}_s)=\{ \hat{\bm{\xi}}_s^{\tau_2} \}$. 
	
	Return to iteration $\tau_2$: Note that, for all $s\in\mathcal{S}$, $\hat{\bm{\lambda}}_s^{\tau_1}\in\hat{\mathcal{P}}_s$ exists, and thus
	\setlength{\arraycolsep}{-0.6em}
	\begin{eqnarray}
		&&LB\geq\bm{c}^{\intercal}\hat{\bm{x}}^{\tau_2}+\sum_{s\in\mathcal{S}}\pi_s\min\eta_s\\
		&&\quad\quad\quad{\rm s.t.}\;\eta_s\geq \bm{d}_s^{\intercal}\bm{y}_s^{\hat{\bm{\lambda}}_s^{\tau_1}},\quad \forall s\in\mathcal{S}\label{}\\
		&&\quad\quad\quad\quad\;\;\bm{B}_s\bm{y}_s^{\hat{\bm{\lambda}}_s^{\tau_1}}\geq\bm{f}_s -\bm{G}_s\hat{\bm{x}}^{\tau_2}-\bm{E}_s\hat{\bm{\xi}}_s^{\tau_2},\quad \forall s\in\mathcal{S}\label{}\\
		&&\quad\quad\quad\quad\;\;\bm{y}_s^{\hat{\bm{\lambda}}_s^{\tau_1}}\in\mathbb{R}^{n_{y}}_+,\quad \forall s\in\mathcal{S}\label{}\\
		&&\quad\quad=\bm{c}^{\intercal}\hat{\bm{x}}^{\tau_2}+\sum_{s\in\mathcal{S}}\pi_s\min_{\bm{y}_s\in\mathcal{Y}_s(\hat{\bm{x}}_2,\hat{\bm{\xi}}^{\tau_2}_s)}\bm{d}_s^{\intercal}\bm{y}_s \\
		&&\quad\quad=\bm{c}^{\intercal}\hat{\bm{x}}^{\tau_2}+\sum_{s\in\mathcal{S}}\pi_s\max_{\bm{\xi}_s\in\Xi_s(\hat{\bm{x}}_2)}\min_{\bm{y}_s\in\mathcal{Y}_s(\hat{\bm{x}}_2,\bm{\xi}_s)}\bm{d}_s^{\intercal}\bm{y}_s\\
		&&\quad\quad\geq UB.
	\end{eqnarray}
	The second equality follows from \eqref{www}. The second inequality holds due to the definition of UB in \eqref{def-ub}. And also, we always have $LB\leq UB$. Thus, $UB=LB$ holds.
	
	When the uniqueness condition is not satisfied, for $\mathfrak{B}_s^{\tau}$ with $s\in\mathcal{S}$, there must exist a component $j$ in the related reduced cost such that $(\bm{r}_{s,\mathfrak{N}_s^{\tau}})^j=0$.	The basic idea is to adjust $(\mathfrak{e}_s^{\hat{\bm{\lambda}}^{\tau}_s})^j$, so that $(\bm{r}_{s,\mathfrak{N}_s^{\tau}})^j>0$ and the uniqueness holds again, and then modify the generated $\mathcal{OU}_s$ with the adjusted $(\mathfrak{e}_s^{\hat{\bm{\lambda}}^{\tau}_s})^j$. In such a situation, the claim's conclusion still holds following the aformentioned proof process. For more details about the modification, please refer to Ref. \cite{zeng2022two}.	
\end{proof}

\begin{proof}[Proof of Theorem \ref{thm-2}]
For each scenario $s\in\mathcal{S}$, the number of bases of $\widetilde{\bm{H}}$ is finite, which is up to ${\binom{n_{\rm \xi}+m_{\rm \xi}}{m_{\rm \xi}}}$. Recall the conclusion in Claim \ref{cla-1}, that PC\&CG exhibits finite convergence to the global optimum ($\varepsilon=0$). With the existence of $|\mathcal{S}|$ scenarios, the iteration number before PC\&CG converges is bounded by ${\binom{n_{\rm \xi}+m_{\rm \xi}}{m_{\rm \xi}}}^{|\mathcal{S}|}$, and thus the iteration complexity of PC\&CG is $\mathcal{O}\left({\binom{n_{\rm \xi}+m_{\rm \xi}}{m_{\rm \xi}}}^{|\mathcal{S}|}\right)$.
\end{proof}

\bibliographystyle{IEEEtran}
\bibliography{ManuscriptBib}

\end{document}